\documentclass[12pt]{amsart}
\usepackage{amsmath}
\usepackage{amsthm}
\usepackage{amstext}
\usepackage{amssymb}
\usepackage{amsfonts}
\usepackage{young}
\usepackage{multicol}
\usepackage{mathrsfs}
\usepackage[mathscr]{euscript}
\usepackage{tikz}
\usetikzlibrary{matrix,arrows,decorations.pathmorphing} 
\usepackage{ytableau}

\newtheorem{theorem}{Theorem}[section]
\newtheorem*{theorem*}{Theorem}
\newtheorem{lemma}[theorem]{Lemma}

\newtheorem{proposition}[theorem]{Proposition}

\newtheorem{definition}[theorem]{Definition}

\theoremstyle{definition}
\newtheorem{example}[theorem]{Example}
\newtheorem{remark}[theorem]{Remark}

\def\m{\mathfrak m}

\def\MPR{{\rm MR}}
\def\mMPR{{\mathfrak{m}}{\rm MR}}
\def\MMPR{{\mathfrak{M}}{\rm MR}}
\def\MNSym{{\mathfrak{M}}{\rm NSym}}
\def\MSym{{\mathfrak{M}}{\rm Sym}}

\def\mQSym{{\mathfrak m}{\rm QSym}}
\def\mSym{{\mathfrak m}{\rm Sym}}
\def\QSym{{\rm QSym}}
\def\NSym{{\rm NSym}}
\def\Sym{{\rm Sym}}

\def \L{{\hat{L}}}
\def\tA{{\hat{\mathcal{A}}}}
\def\K{\hat{K}}
\def\tJ{\tilde{\mathcal{J}}}
\def\tK{\hat{K}}

\DeclareMathAlphabet{\mathpzc}{OT1}{pzc}{m}{it}

\title{
Antipode formulas for some combinatorial Hopf algebras
}

\numberwithin{equation}{section}

\begin{document}

\author{Rebecca Patrias}
\address{\hspace{-.3in} LaCIM, Universit\'e du Qu\'ebec \`a Montr\'eal,
Montr\'eal (Qu\'ebec), Canada}
\email{patriasr@lacim.ca}\thanks{R.P. was supported by NSF grant DMS-1148634.}

\date{\today
}

\thanks{
}

\subjclass{
Primary
05E05, 
}

\keywords{}

\begin{abstract}
Motivated by work of Buch on set-valued tableaux in relation to the $K$-theory of the Grassmannian, Lam and Pylyavskyy studied six 
combinatorial Hopf algebras that can be thought of as $K$-theoretic analogues of the Hopf algebras of symmetric functions, 
quasisymmetric functions, noncommutative symmetric functions, and of the Malvenuto-Reutenauer Hopf algebra of permutations. They described the 
bialgebra structure in all cases that were not yet known but left open the question of finding explicit formulas for the antipode maps. We 
give combinatorial formulas for the antipode map for the $K$-theoretic analogues of the symmetric functions, quasisymmetric functions, and noncommutative symmetric functions.
\end{abstract}

\ \vspace{-.1in}

\maketitle

\section{Introduction}

A Hopf algebra is a structure that is both an associative algebra with unit and a coassociative 
coalgebra with counit. 
The algebra and coalgebra structures are compatible, which makes it a bialgebra. 
To be a Hopf algebra, a bialgebra must have a special anti-endomorphism called the antipode,
which must satisfy certain properties. 

Hopf algebras arise naturally in combinatorics. Notably, the symmetric functions ($\Sym$), 
quasisymmetric functions ($\QSym$), noncommutative symmetric functions
($\NSym$), and the Malvenuto-Reutenauer algebra of permutations ($\MPR$) are Hopf algebras,
 which can be arranged as shown in Figure~\ref{fig:combhopf}.
 
\begin{figure}[h!]
\begin{center}
\begin{tikzpicture}[scale=1.5]
\node (A) at (-1,0) {$\Sym$};
\node (B) at (1,0) {$\Sym$};
\node (C) at (-1,1) {$\NSym$};
\node (D) at (1,1) {$\QSym$};
\node (E) at (-1, 2) {$\MPR$};
\node (F) at (1,2) {$\MPR$};
\path[->>]
(F) edge node[right]{$ $} (D)
(C) edge node[left]{$ $} (A);
\path[right hook->]
(B) edge node[right]{$ $} (D)
(C) edge node[right]{$ $} (E);
\draw (A) -- (B)
(C) -- (D)
(E) -- (F);
\end{tikzpicture}
\end{center}
\caption{Diagram of combinatorial Hopf algebras}
\label{fig:combhopf}
\end{figure}
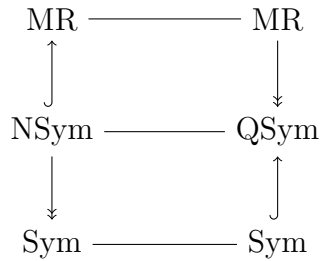

Through the work of Lascoux and Sch\"{u}tzengerger \cite{lascoux1structure}, Fomin and Kirillov \cite{fomin1996yang}, and Buch \cite{buch2002littlewood}, symmetric 
functions known as stable Grothendieck polynomials were discovered and given a combinatorial 
interpretation in terms of set-valued tableaux.
They originated from Grothendieck polynomials, which serve as representatives of $K$-theory classes 
of structure sheaves of Schubert varieties.
The stable Grothendieck polynomials play the role of Schur functions in the $K$-theory of Grassmannians. 
They also determine a $K$-theoretic
analogue of the symmetric functions, which we call the multi-symmetric functions and denote $\mSym$. 

In \cite{lam2007combinatorial}, Lam and Pylyavksyy extend the definition of $P$-partitions to create $P$-set-valued partitions,
 which they use to define a new $K$-theoretic analogue
of the Hopf algebra of quasisymmetric functions called the Hopf algebra of multi-quasisymmetric functions. 
The entire diagram may be extended to give the diagram in Figure~\ref{fig:Kcombhopf}.

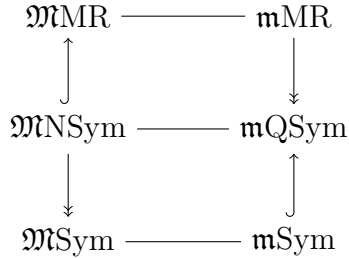
\begin{figure}[h!]
\begin{center}
\begin{tikzpicture}[scale=1.5]
\node (A) at (-1,0) {$\MSym$};
\node (B) at (1,0) {$\mSym$};
\node (C) at (-1,1) {$\MNSym$};
\node (D) at (1,1) {$\mQSym$};
\node (E) at (-1, 2) {$\MMPR$};
\node (F) at (1,2) {$\mMPR$};
\path[->>]
(F) edge node[right]{$ $} (D)
(C) edge node[left]{$ $} (A);
\path[right hook->]
(B) edge node[right]{$ $} (D)
(C) edge node[right]{$ $} (E);
\draw (A) -- (B)
(C) -- (D)
(E) -- (F);
\end{tikzpicture}
\end{center}
\caption{Diagram of $K$-theoretic combinatorial Hopf algebras of Lam and Pylyavskyy}
\label{fig:Kcombhopf}
\end{figure}

Using Takeuchi's formula \cite{takeuchi1971free}, they give a formula for the antipode for $\MMPR$ but leave open the
 question of an antipode for the remaining Hopf algebras.
In this paper, we give the first combinatorial formulas for the antipode maps of $\MNSym$, $\mQSym$, $\mSym$, and $\MSym$. 

\begin{remark}
As there is no sufficiently nice combinatorial formula for the antipode map in the Malvenuto-Reutenauer Hopf algebra of permutations, we do not attempt a formula for its $K$-theoretic analogues: $\MMPR$ and $\mMPR$.
\end{remark}

After a brief introduction to Hopf algebras, we introduce the Hopf algebra $\mQSym$ in Section~\ref{mQSym}. 
Next, we introduce  $\MNSym$ in Section~\ref{MNSym}. We present results concerning the antipode map in $\MNSym$ and $\mQSym$, namely
Theorems~\ref{MNSymAntipode} and \ref{thm:mQSymAntipode}. In Section~\ref{sec:Lhat}, we present an additional 
basis for $\mQSym$,
give analogues of results in \cite{lam2007combinatorial} for this new basis, and give an antipode formula in $\mQSym$ involving the new basis
in Theorem~\ref{LhatAntipode}.
Lastly, we introduce the Hopf algebras of multi-symmetric functions, $\mSym$, and of Multi-symmetric 
functions, $\MSym$ in Sections~\ref{mSym} and \ref{MSym}. We end with Theorems \ref{thm:mSymAntipode}, \ref{thm:MSymAntipode}, and \ref{thm:S(G)}, which describe antipode maps in these spaces.
 
\section{Hopf algebra basics}\label{sec:Hopf}
\subsection{Algebras and coalgebras}

First we build a series of definitions leading to the definition of a Hopf algebra. 
For further reading, we recommend \cite{joni1979coalgebras, montgomery1993hopf, grinberg2014hopf, sweedler1969hopf}.

In this section, $k$ will usually denote a field, although it may also be a commutative ring. 
In all later sections we take $k=\mathbb{Z}$. All tensor products are taken over $k$.

\begin{definition}\label{algebra}
 An associative k-algebra $A$ is a $k$-vector space with associative operation $m:A\otimes A \to A$ (the product)
and unit map $\eta:k \to A$ with $\eta(1_k)=1_A$ such that the following diagrams commute: 

\begin{center}
\begin{minipage}{0.45\textwidth}
\begin{tikzpicture}[scale=1.5]
\node (A) at (0,2) {$A\otimes A \otimes A$};
\node (B) at (2,2) {$A\otimes A$};
\node (C) at (0,0) {$A \otimes A$};
\node (D) at (2,0) {$A$};
\path[->,font=\scriptsize,>=angle 90]
(A) edge node[above]{$1\otimes m$} (B)
(A) edge node[right]{$m \otimes 1$} (C)
(B) edge node[right]{$m$} (D)
(C) edge node[above]{$m$} (D);
\end{tikzpicture}
\end{minipage}
\begin{minipage}{0.45\textwidth}
\begin{tikzpicture}[scale=1.5]
\node (A) at (0,0) {$A\otimes A$};
\node (B) at (-2,2) {$k \otimes A$};
\node (C) at (0,2) {$ A$};
\node (D) at (2,2) {$A\otimes k$};
\path[->,font=\scriptsize,>=angle 90]
(B) edge node[left]{$\eta \otimes 1$} (A)
(A) edge node[right]{$m$} (C)
(D) edge node[right]{$1\otimes \eta$} (A)
(C) edge node[above]{$\cong$} (D)
(D) edge node[above]{$\cong$} (C)
(C) edge node[above]{$\cong$} (B)
(B) edge node[above]{$\cong$} (C);
\end{tikzpicture}
\end{minipage}
\end{center}
where we take the isomorphisms sending $a\otimes k$ to $ak$ and $k\otimes a$ to $ka$.
\end{definition}

The first diagram tells us that $m$ is an associative product and the second that $\eta(1_k)=1_A$.

\begin{definition} A co-associative coalgebra C is a k-vector space with $k$-linear map
 $\Delta:C\to C\otimes C$ (the coproduct) and a counit $\epsilon:C \to k$ such that the following diagrams commute. 
\begin{center}
\begin{minipage}{0.45\textwidth}
\begin{tikzpicture}[scale=1.5]
\node (A) at (0,2) {$C$};
\node (B) at (2,2) {$C\otimes C$};
\node (C) at (0,0) {$C \otimes C$};
\node (D) at (2,0) {$C\otimes C\otimes C$};
\path[->,font=\scriptsize,>=angle 90]
(A) edge node[above]{$\Delta$} (B)
(A) edge node[right]{$\Delta$} (C)
(C) edge node[above]{$1\otimes\Delta$} (D)
(B) edge node[right]{$\Delta\otimes 1$} (D);
\end{tikzpicture}
\end{minipage}
\begin{minipage}{0.45\textwidth}
\begin{tikzpicture}[scale=1.5]
\node (A) at (0,0) {$C\otimes C$};
\node (B) at (-2,2) {$k \otimes C$};
\node (C) at (0,2) {$ C$};
\node (D) at (2,2) {$C\otimes k$};
\path[->,font=\scriptsize,>=angle 90]
(A) edge node[left]{$\epsilon \otimes 1$} (B)
(C) edge node[right]{$\Delta$} (A)
(A) edge node[right]{$1\otimes \epsilon$} (D)
(D) edge node[above]{$\cong$} (C)
(C) edge node[above]{$\cong$} (D)
(B) edge node[above]{$\cong$} (C)
(C) edge node[above]{$\cong$} (B);
\end{tikzpicture}
\end{minipage}
\end{center}
\end{definition}

The diagram on the left indicates that $\Delta$ is \textit{co-associative}. 
Note that these are the same diagrams as in the Definition~\ref{algebra} with all of the arrows reversed. 

It is often useful to think of the product as a way to combine two elements of an algebra and 
to think of the coproduct as a sum over ways to split a coalgebra element into two pieces.
When discussing formulas involving $\Delta$, we will use Sweedler notation as shown below:
$$\Delta(c)=\sum_{(c)}{c_1 \otimes c_2}=\sum{c_1 \otimes c_2}.$$  This is a common convention 
that will greatly simplify our notation.

\begin{example}\label{ex:shuffle1}
To illustrate the concepts just defined, we give the example of the shuffle algebra, 
which is both an algebra and coalgebra. Let $I$ be an alphabet and $\bar{I}$ be the set of words on $I$. We declare that the elements of $\bar{I}$ form a $k$-basis for the shuffle algebra.

Given two words $a=a_1a_2\cdots a_t$ and $b=b_1 b_2 \cdots b_n$ in $\bar{I}$, define their product, $m(a\otimes b)$, to be the \textit{shuffle product} of
$a$ and $b$. That is, $m(a\otimes b)$ is the sum of all $\binom{t+n}{n}$ ways to interlace the 
two words while maintaining the relative order of the letters in each word. 
For example, $$m(a_1a_2 \otimes b_1)=a_1a_2b_1+a_1b_1a_2+b_1a_1a_2.$$ We may then extend by 
linearity. It is not hard to see that this multiplication is associative. 

The unit map for the shuffle algebra is defined by $\eta(1_k)=\emptyset$, where $\emptyset$ is the empty word. 
Note that $m(a\otimes \emptyset)=m(\emptyset \otimes a)=a$ for any word $a$.

For a word $a=a_1a_2\cdots a_t$ in $\bar{I}$, we define 
$$\Delta(a)=\sum_{i=0}^t{a_1a_2\cdots a_i \otimes a_{i+1}a_{i+2}\cdots a_t}$$ and call this the 
\textit{cut coproduct} of $a$. For example, given a word $a=a_1a_2$, 
$$\Delta(a)=\emptyset \otimes a_1a_2 + a_1\otimes a_2 + a_1a_2\otimes\emptyset.$$

The counit map is defined by letting $\epsilon$ take the coefficient of the empty word. 
Hence for any nonempty $a \in \bar{I}$, $\epsilon(a)=0$. 
\end{example}

\subsection{Morphisms and bialgebras}

The next step in defining a Hopf algebra is to define a bialgebra. For this, we need a notion of compatibility 
of maps of an algebra $(m,\eta)$ and maps of a coalgebra $(\Delta,\epsilon)$.
With this as our motivation, we introduce the following definitions.

\begin{definition}
 If $A$ and $B$ are k-algebras with multiplication $m_A$ and $m_B$ and unit maps $\eta_A$ and $\eta_B$, 
respectively, then a k-linear map
$f:A \to B$ is an algebra morphism if $f \circ m_A=m_B\circ (f\otimes f)$ and $f\circ \eta_A =\eta_B$.
\end{definition}

\begin{definition}
 Given k-coalgebras $C$ and $D$ with comultiplication and counit 
$\Delta_C$, $\epsilon_C$, $\Delta_D$, and $\epsilon_d$, 
k-linear map $g: C \to D$ is a coalgebra morphism if 
$\Delta_D \circ g=(g\otimes g)\circ \Delta_C$ and $\epsilon_D\circ g = \epsilon_C$.
\end{definition}

Given two $k$-algebras $A$ and $B$, their tensor product $A\otimes B$ is 
also a $k$-algebra with $m_{A\otimes B}$ defined to be the composite of

\begin{center}
\begin{tikzpicture}[scale=1.5]
\node (A) at (-3,0) {$A\otimes B \otimes A\otimes B$};
\node (B) at (0,0) {$A\otimes A \otimes B \otimes B$};
\node (C) at (2.5,0) {$A \otimes B,$};
\path[->,font=\scriptsize,>=angle 90]
(A) edge node[above]{$1\otimes T \otimes 1$} (B)
(B) edge node[above]{$m_A \otimes m_B$} (C);
\end{tikzpicture}
\end{center}
where $T(b\otimes a)=a\otimes b$. For example, we have $$m_{A\otimes B}((a\otimes b)\otimes(a'\otimes b'))=m_A(a\otimes a')\otimes m_B(b\otimes b').$$
The unit map in $A\otimes B$, $\eta_{A\otimes B}$, is given by the composite 

\begin{center}
\begin{tikzpicture}[scale=1.5]
\node (A) at (-1,0) {$k$};
\node (B) at (0,0) {$k\otimes k$};
\node (C) at (1.5,0) {$A \otimes B.$};
\path[->,font=\scriptsize,>=angle 90]
(A) edge node[above]{$ $} (B)
(B) edge node[above]{$\eta_A \otimes \eta_B$} (C);
\end{tikzpicture}
\end{center}

Similarly, given two coalgebras $C$ and $D$, their tensor product $C\otimes D$ is a coalgebra with 
$\Delta_{C\otimes D}$ the composite of \\

\begin{center}
\begin{tikzpicture}[scale=1.5]
\node (A) at (-2.5,0) {$C\otimes D$};
\node (B) at (0,0) {$C\otimes C \otimes D \otimes D$};
\node (C) at (3,0) {$C \otimes D \otimes C \otimes D,$};
\path[->,font=\scriptsize,>=angle 90]
(A) edge node[above]{$\Delta_C\otimes \Delta_D$} (B)
(B) edge node[above]{$1\otimes T \otimes 1$} (C);
\end{tikzpicture}
\end{center}
and the counit $\epsilon_{A\otimes B}$ is the composite \\

\begin{center}
\begin{tikzpicture}[scale=1.5]
\node (A) at (-1.5,0) {$C\otimes D$};
\node (B) at (0,0) {$k\otimes k$};
\node (C) at (1,0) {$k$.};
\path[->,font=\scriptsize,>=angle 90]
(A) edge node[above]{$\epsilon_C \otimes \epsilon_D$} (B)
(B) edge node[above]{$ $} (C);
\end{tikzpicture}
\end{center}

\begin{definition}
Given A that is both a k-algebra and a k-coalgebra, we call A a k-bialgebra if $(\Delta, \epsilon)$ 
are morphisms for the algebra structure
$(m, \eta)$ or equivalently, if $(m,\eta)$ are morphisms for the coalgebra structure $(\Delta, \epsilon)$.
\end{definition}

\begin{example}\label{ex:shuffle2}
The shuffle algebra is a bialgebra. We can see, for example, that
\begin{eqnarray}
 \Delta \circ m_A (a_1\otimes b_1)&=&\Delta(a_1b_1+b_1a_1)\nonumber\\ 
&=& \emptyset\otimes a_1b_1+a_1\otimes b_1+a_1b_1\otimes\emptyset+\emptyset\otimes b_1a_1+b_1\otimes a_1+b_1a_1\otimes\emptyset\nonumber\\
&=& \emptyset\otimes (a_1b_1+b_1a_1)+b_1\otimes a_1 + a_1\otimes b_1 + (a_1b_1+b_1a_1)\otimes \emptyset\nonumber\\
&=& m_A(\emptyset \otimes \emptyset)\otimes m_A(a_1\otimes b_1)+m_A(\emptyset\otimes b_1)\otimes m_A(a_1\otimes\emptyset)+m_A(a_1\otimes \emptyset)\otimes m_A(\emptyset\otimes b_1)\nonumber\\
& & +m_A(a_1\otimes b_1)\otimes m_A(\emptyset\otimes\emptyset)\nonumber\\
&=& m_{A\otimes A} ((\emptyset\otimes a_1+a_1\otimes\emptyset)\otimes(\emptyset\otimes b_1+b_1\otimes\emptyset))\nonumber\\
&=& m_{A\otimes A} \circ (\Delta(a_1)\otimes \Delta(b_1)).\nonumber
\end{eqnarray}
This is evidence that the coproduct, $\Delta$, is an algebra morphism.
\end{example}

\subsection{The antipode map}

A Hopf algebra is a bialgebra equipped with an additional map called the antipode map. 
On our way to defining the antipode map,
we must first introduce an algebra structure on $k$-linear algebra maps that take coalgebras to algebras.

\begin{definition} 
 Given coalgebra C and algebra A, we form an associative algebra structure on the set of $k$-linear maps from $C$ to $A$, $Hom_k(C,A)$, called the convolution
algebra as follows: for $f$ and $g$ in $Hom_k(C,A)$, define the product, $f * g$, by $$(f * g)(c)
=m \circ (f\otimes g)\circ \Delta(c)=\sum f(c_1)g(c_2),$$ where
$\Delta(c)=\sum c_1 \otimes c_2$.
\end{definition}

Note that $\eta \circ \epsilon$ is the two-sided identity element for $*$ using this product. We can easily see this 
in the shuffle algebra from Examples~\ref{ex:shuffle1} and \ref{ex:shuffle2}
if we remember that $(\eta \circ \epsilon)(a)=\eta(0)=0$ for all words $a \neq \emptyset$. 
Let $c$ be a word in the shuffle algebra, then
$$(f * (\eta \circ \epsilon))(c) = \sum f(c_1)(\eta \circ \epsilon)(c_2)=f(c)=\sum(\eta \circ \epsilon)(c_1)f(c_2)=((\eta \circ \epsilon) * f)(c)$$ 
because $c_1=c$ when $c_2=\emptyset$ and $c_2=c$ when $c_1=\emptyset$.

If we have a bialgebra $A$, then we can consider this convolution structure to be on $End_k(A):= Hom_k(A,A).$

\begin{definition}\label{convolutionalinverse}
 Let $(A,m, \eta, \Delta, \epsilon)$ be a bialgebra. Then $S \in End_k(A)$ is called an antipode for bialgebra A
if $$id_A * S = S * id_A = \eta \circ \epsilon,$$ where $id_A:A \to A$ is the identity map.
\end{definition}

In other words, the endomorphism $S$ is the two-sided inverse for the identity map $id_A$ under the 
convolution product. Equivalently, if $\Delta(a)=\sum{a_1\otimes a_2}$, 
$$(S * id_A)(a)=\sum{S(a_1)a_2}=\eta(\epsilon(a))=\sum{a_1S(a_2})=(id_A * S).$$
Because we have an associative algebra, this means that if an antipode exists, then it is unique.

\begin{example}
We define the antipode of a word in the shuffle algebra by 
\[
 S(a_1a_2\cdots a_t)=(-1)^ta_ta_{t-1}\cdots a_2 a_1
\]
and extend by linearity. We can see an example of the defining property by computing
\begin{eqnarray}
(id * S)(a_1a_2) &=& m(id(\emptyset)\otimes S(a_1a_2))+m(id(a_1)\otimes S(a_2))+m(id(a_1a_2)\otimes S(\emptyset))\nonumber\\
&=& -a_2a_1-m(a_1\otimes a_2)+a_1a_2\nonumber \\
&=& -a_2a_1-(a_1a_2+a_2a_1)+a_1a_2\nonumber\\
&=& 0 \nonumber\\
&=& \eta(\epsilon(a_1a_2)).\nonumber
\end{eqnarray}
\end{example}

We end this section with two useful properties that we use in later sections. The first is a well-known property of the antipode map 
for any Hopf algebra.

\begin{proposition}\label{antiendomorphism}
Let $S$ be the antipode map for Hopf algebra $A$. Then $S$ is an algebra anti-endomorphism: 
$S(1)=1$, and $S(ab)=S(b)S(a)$ for all 
$a$, $b$ in $A$.
\end{proposition}

The second property allows us to translate antipode formulas between certain Hopf algebras.

\begin{lemma}\label{lem:pairingantipode}
Suppose we have two bialgebra bases, $\{A_\lambda\}$ and $\{B_\mu\}$, that are 
dual under a pairing and such that the structure constants for the 
product of the first basis are the structure constants for the coproduct of the second basis and vice versa. 
In other words, $\langle A_\lambda, B_\mu \rangle = \delta_{\lambda,\mu};$ $A_\lambda A_\mu = \displaystyle\sum_\nu f^\nu_{\lambda,\mu} A_\nu$ and 
$\Delta(B_\lambda)=\displaystyle\sum_{\mu,\nu}f^\lambda_{\mu,\nu} B_\mu\otimes B_\nu$; and  
$\Delta(A_\lambda)=\displaystyle\sum_{\mu,\nu}h^\lambda_{\mu,\nu} A_\mu\otimes A_\nu$ and $B_\lambda B_\mu = 
\displaystyle\sum_\nu h^\nu_{\lambda,\mu} B_\nu$. 
If $$S(A_\lambda)=\displaystyle\sum_{\mu} e_{\lambda,\mu} A_\mu$$ for $S$ satisfying $0=\displaystyle\sum_{}h_{\mu,\nu}^\lambda S(A_\mu)A_\nu$,
then $$S(B_\mu)=\displaystyle\sum_{\lambda} e_{\lambda,\mu} B_\lambda$$ satisfies $\displaystyle\sum_{}f^\lambda_{\mu,\nu}S(B_\mu)B_\nu=0$.
\end{lemma}

\begin{proof} Indeed, 
\begin{eqnarray}
 \left\langle \displaystyle\sum_{\mu,\nu}f^\lambda_{\mu,\nu}S(B_\mu)B_\nu, A_\tau \right\rangle &=&  
 \left\langle \displaystyle\sum_{\mu,\nu,\gamma}f^\lambda_{\mu,\nu,\gamma}k_{\gamma,\mu} B_\gamma B_\nu, A_\tau \right\rangle\nonumber \\
 &=& \left\langle \displaystyle\sum_{\mu,\nu,\gamma, \rho}f^\lambda_{\mu,\nu,\gamma}k_{\gamma,\mu} h^\rho_{\gamma,\nu} B_\rho, A_\tau \right\rangle\nonumber \\
 &=& \displaystyle\sum_{\mu,\nu,\gamma}f^\lambda_{\mu,\nu,\gamma}k_{\gamma,\mu} h^\tau_{\gamma,\nu} \nonumber\\
 &=& \left\langle B_\lambda, \displaystyle\sum_{\rho, \mu, \nu, \gamma}h_{\gamma,\nu}^\tau k_{\gamma, \mu} f_{\mu,\nu}^\rho A_\rho \right \rangle \nonumber \\
  &=& \left\langle B_\lambda, \displaystyle\sum_{\mu, \nu, \gamma}h_{\gamma,\nu}^\tau k_{\gamma, \mu} A_\mu A_\nu \right \rangle \nonumber \\
    &=& \left\langle B_\lambda, \displaystyle\sum_{\nu, \gamma}h_{\gamma,\nu}^\tau S(A_\gamma) A_\nu \right \rangle \nonumber \\
    &=& 0 \nonumber
\end{eqnarray}
by assumption. 
\end{proof}

\section{The Hopf algebra of multi-quasisymmetric functions}\label{mQSym}

In what follows, we say that a set $\{A_\lambda\}$ \textit{continuously spans} space $A$ 
if everything in $A$ can be written 
as a (possibly infinite) linear combination of $A_\lambda$'s. Here, we assume that $\{A_\lambda\}$ comes with a natural filtration and that each
filtered component is finite. Then we may talk about continuous
span with respect to the topology induced by the filtration. A \textit{continuous basis} for $A$ allows elements to be written as 
arbitrary linear combinations of the basis elements. We say that a linear function $f:A \to A$ is \textit{continuous} if it respects 
arbitrary linear combinations of elements in $A$.

We next introduce the multi-quasisymmetric functions, $\mQSym$. It may be useful for the reader to be familiar with the Hopf algebra of quasisymmetric functions, specifically the basis of fundamental quasisymmetric functions. We recommend \cite{stanley1999enumerative} for background reading. 

\subsection{$(P,\theta)$-set-valued partitions}\label{sec:PthetaPartitions}
Following \cite{lam2007combinatorial}, we define $\mQSym$, the \textit{Hopf algebra of multi-quasisymmetric functions}, by defining the continuous basis of multi-fundamental quasisymmetric functions, ${\tilde{L}_\alpha}$. Let $[n]=\{1,2,\ldots,n\}$. 
We start with a finite poset $P$ with $n$ elements and a bijective labeling $\theta:P \to [n]$. Let $\mathbb{\tilde P}$ denote the set of nonempty, finite subsets
of the positive integers. If $a\in\mathbb{\tilde P}$ and $b \in \mathbb{\tilde P}$ are two such subsets, we say that $a < b$ if $max(a)<min(b)$. Similarly, 
$a \leq b$ if $max(a) \leq min(b)$.

We next define the $(P,\theta)$-set-valued partitions. The definition is almost identical to that of the more well-known $(P,\theta)$-partitions except that we will assign a nonempty, finite subset of positive integers to each element of the poset instead of assigning a single positive integer. We recommend \cite{stanley1999enumerative} for further reading on $(P,\theta)$-partitions.

\begin{definition}
 Let $(P,\theta)$ be a poset with a bijective labeling. A $(P,\theta)$-set-valued partition is a map $\sigma:P\to \mathbb{\tilde P}$ such that for each covering relation
$s\lessdot t$ in $P$,
\begin{enumerate}
\item $\sigma(s) \leq \sigma(t)$ if $\theta(s) < \theta(t)$,
\item $\sigma(s) < \sigma(t)$ if $\theta(s) > \theta(t)$.
\end{enumerate}
\end{definition}

\begin{example}
 The diagram on the left shows an example of a poset $P$ with a bijective labeling $\theta$. We identify elements of $P$ with their labeling. The diagram on the right shows a valid $(P,\theta)$-set-valued partition $\sigma$.
Note that since $3<2$ in the poset, we must have the strict inequality $max(\sigma(3))=6< min(\sigma(2))=30$.\\
\begin{minipage}{0.5\textwidth}
\begin{center}
\begin{tikzpicture}
  \node (a) at (2,2) {$2$};
  \node (b) at (-2,0) {$4$};
  \node (c) at (2,0) {$3$};
  \node (d) at (0,-2) {$1$};
  \draw (d) -- (b)
(d) -- (c) -- (a);
\end{tikzpicture}
\end{center}
\end{minipage}
\begin{minipage}{0.5\textwidth}
\begin{center}
\begin{tikzpicture}
  \node (a) at (2,2) {$\sigma(2)=\{30, 31, 32\}$};
  \node (b) at (-2,0) {$\sigma(4)=\{7,100\}$};
  \node (c) at (2,0) {$\sigma(3)=\{6\}$};
  \node (d) at (0,-2) {$\sigma(1)=\{1,3,6\}$};
  \draw (d) -- (b)
(d) -- (c) -- (a);
\end{tikzpicture}
\end{center}
\end{minipage}
\end{example}

We denote the set of all $(P,\theta)$-set-valued partitions for given poset $P$ by $\tilde{\mathcal{A}}(P,\theta)$. For each element $i$ in $P$, let 
$\sigma^{-1}(i)=\{x \in P \mid i\in\sigma(x)\}$. Now define $\tilde K _{P,\theta}\in\mathbb{Z}[[x_1,x_2,\ldots ]]$ by $$\tilde K_{P,\theta}=\displaystyle
\sum_{\sigma\in\tilde{\mathcal{A}}(P,\theta)}{x_1^{\#\sigma^{-1}(1)}x_2^{\#\sigma^{-1}(2)}\cdots}.$$ 
For example, the $(P,\theta)$-set-valued partition in the previous example contributes $$x_1x_3x_6^2x_7x_{30}x_{31}x_{32}x_{100}$$ to $\tilde K_{P,\theta}$.
Note that $\tilde K _{P,\theta}$ will be of unbounded degree for any nonempty poset $P$. 

We may also consider a Young diagram $\lambda$ as a poset in the natural way as follows. Let $P$ be the poset of squares in the Young diagram of a partition $\lambda = (\lambda_1,\lambda_2,\ldots,\lambda_t)$ and $\theta_s$ be the bijective labeling of $P$
obtained from labeling $P$ in row reading order, i.e. from left to right the bottom row of $\lambda$ is labeled $1,2,\ldots,\lambda_t$, the next row up is labeled 
$\lambda_t+1,\ldots, \lambda_t+\lambda_{t-1}$ and so on. We may thus refer to the function $\tilde{K}_{\lambda,\theta_s}$. We will see this idea next in Example~\ref{ex:Lhatbijection}.

\subsection{The multi-fundamental quasisymmetric functions}\label{sec:multifunquasis}
A composition of $n$ is an ordered arrangement of positive integers that sum to $n$. 
For example, $(3)$, $(1,2)$, $(2,1)$, and $(1,1,1)$ are all of the compositions of $3$.

If $S=\{s_1,s_2,\ldots ,s_k\}$ is a subset of $[n-1]$, we associate a composition, $\mathcal{C}(S)$, 
to $S$ by $\mathcal{C}(S)=\{s_1, s_2-s_1,s_3-s_2,\ldots ,n-s_k\}.$ To composition $\alpha$ of $n$,
we associate $S_\alpha \subset [n-1]$ by letting $S_\alpha=\{\alpha_1, \alpha_1+\alpha_2, \ldots , 
\alpha_1+\alpha_2+\ldots +\alpha_{k-1}\}.$ We may extend this correspondence
to permutations by letting $\mathcal{C}(w)=\mathcal{C}(Des(w))$, where $w \in \mathfrak{S}_n$ and $Des(w)$ 
is the descent set of $w$. For example,
if $S=\{1,4,5\}\subset[6-1]$, $\mathcal{C}(S)=(1, 4-1, 5-4, 6-5)=(1,3,1,1)$. Conversely, 
given composition $\alpha=(1,3,1,1)$, $S_\alpha=\{1, 1+3,1+3+1\}=\{1,4,5\}$. For $w=132\in\mathfrak{S}_3$, 
$Des(w)=\{2\}$ and $\mathcal{C}(w)=(2,1)$. Given a composition $\alpha$ of $n$, we write $w_\alpha$ to denote  any permutation in $\mathfrak{S}_n$ with $\mathcal{C}(w_\alpha)=\alpha$.

We may now define the multi-fundamental quasisymmetric function $\tilde L_\alpha$ indexed by composition $\alpha$.

\begin{definition}
Let $P$ be a finite chain $p_1<p_2<\ldots<p_k$, $w\in \mathfrak{S}_k$ a permutation, and $\mathcal{C}(w)=\alpha$ the composition
of $n$ associated to the descent set of $w$. We label $P$ using $w$ with $\theta(p_i)=w_i$. Then $$\tilde L_\alpha=\tilde{K}_{(P,w)}=\displaystyle
\sum_{\sigma\in \tilde{\mathcal{A}}(P,w)}{x_1^{\#\sigma^{-1}(1)}x_2^{\#\sigma^{-1}(2)}\cdots}.$$ 
\end{definition}

It is easy to see that $\tilde K_{(P,w)}$ depends only on $\alpha$. Note that this is an infinite sum of unbounded degree. The sum of the lowest degree terms in $\tilde L_\alpha$ gives $L_\alpha$, the 
fundamental quasisymmetric function in $\QSym$.
 
\begin{example}
 Let $\alpha=(2,1)$ and $w_\alpha = 231$. We consider all $(P,w_\alpha)$-set-valued partitions on the chain shown below 
 on the far left. The seven images to its right show examples of images of the map $\sigma$.
 
\begin{center}
\begin{minipage}{.08\textwidth}
\begin{tikzpicture}
  \node (a) at (0,0) {$2$};
  \node (b) at (0,2) {$3$};
  \node (c) at (0,4) {$1$};
  \draw (a) -- (b) -- (c);
\end{tikzpicture}
\end{minipage}
\begin{minipage}{.08\textwidth}
\begin{tikzpicture}
  \node (a) at (0,0) {$\{1\}$};
  \node (b) at (0,2) {$\{1\}$};
  \node (c) at (0,4) {$\{2\}$};
  \draw (a) -- (b) -- (c);
\end{tikzpicture}
\end{minipage}
\begin{minipage}{.08\textwidth}
\begin{tikzpicture}
  \node (a) at (0,0) {$\{1\}$};
  \node (b) at (0,2) {$\{1\}$};
  \node (c) at (0,4) {$\{3\}$};
  \draw (a) -- (b) -- (c);
\end{tikzpicture}
\end{minipage}
\begin{minipage}{.08\textwidth}
\begin{tikzpicture}
  \node (a) at (0,0) {$\{1\}$};
  \node (b) at (0,2) {$\{2\}$};
  \node (c) at (0,4) {$\{3\}$};
  \draw (a) -- (b) -- (c);
\end{tikzpicture}
\end{minipage}
\begin{minipage}{.08\textwidth}
\begin{tikzpicture}
  \node (a) at (0,0) {$\{1,2\}$};
  \node (b) at (0,2) {$\{2\}$};
  \node (c) at (0,4) {$\{3,4\}$};
  \draw (a) -- (b) -- (c);
\end{tikzpicture}
\end{minipage}
\begin{minipage}{.08\textwidth}
\begin{tikzpicture}
  \node (a) at (0,0) {$\{1,2\}$};
  \node (b) at (0,2) {$\{2,3\}$};
  \node (c) at (0,4) {$\{4\}$};
  \draw (a) -- (b) -- (c);
\end{tikzpicture}
\end{minipage}
\begin{minipage}{.08\textwidth}
\begin{tikzpicture}
  \node (a) at (0,0) {$\{5,6,7\}$};
  \node (b) at (0,2) {$\{7,100\}$};
  \node (c) at (0,4) {$\{101\}$};
  \draw (a) -- (b) -- (c);
\end{tikzpicture}
\end{minipage}
\begin{minipage}{.08\textwidth}
\begin{tikzpicture}
  \node (a) at (0,0) {$\{5,6,7\}$};
  \node (b) at (0,2) {$\{7\}$};
  \node (c) at (0,4) {$\{100,101\}$};
  \draw (a) -- (b) -- (c);
\end{tikzpicture}
\end{minipage}
\end{center} 
Using the examples above, we see that 
\[
\tilde L_{(2,1)}=x_1^2x_2+x_1^2x_3+x_1x_2x_3+2x_1x_2^2x_3x_4+2x_5x_6x_7^2x_{100}x_{101}+\ldots,
\]
an infinite sum of unbounded degree.
\end{example}

For the following definition, we order $\tilde{\mathbb{P}}$ as in Section~\ref{sec:PthetaPartitions}. Namely, for subsets $a$ and $b$, we say $a \leq b$ if $max(a)\leq min(b)$ and $a<b$ if $max(a)< min(b)$, though we shall only need the latter notion.

\begin{definition}
Given a poset $P$ with $n$ elements, a \textit{linear multi-extension} of $P$ by $[N]$ is a map $e: P \to 2^{[N]}$ for some $N\geq n$ such that 
\begin{enumerate}
\item $e(x)<e(y)$ if $x<y$ in $P$,
\item each $i\in [N]$ is in $e(x)$ for exactly one $x\in P$, and
\item no set $e(x)$ contains both $i$ and $i+1$ for any $i$.
\end{enumerate}
\end{definition}

For $e$ any linear multi-extension of $P$ by $[N]$ and any $i\in[N]$, let $e^{-1}(i)$ denote the unique element $p$ of $P$ such that $i\in e(p)$. Note that $e^{-1}(\{i\})$ may be empty while $e^{-1}(i)$ always contains exactly one element of $P$.
We then define the \textit{multi-Jordan-Holder set} $\tilde{\mathcal{J}}(P,\theta)=\cup_N\tilde{\mathcal{J}_N}(P,\theta)$ to be the union of the sets  
$$ \tilde{\mathcal{J}}_N (P,\theta)=\{ \theta(e^{-1}(1))\theta(e^{-1}(2))\cdots \theta(e^{-1}(N))\mid e\text{ is a linear multi-extension of $P$ by $[N]$}\}.$$
Note that elements in $ \tilde{\mathcal{J}}_N (P,\theta)$ are $\mathfrak{m}$-permutations---pronounced ``multi-permutations''--- \cite{lam2007combinatorial} of $[n]$ with $N$ letters, where we define an $\mathfrak{m}$-permutation
of $[n]$ to be a word in the alphabet $1,2,\ldots,n$ such that no two consecutive letters are equal.

\begin{example}
Consider again the labeled poset below.

\begin{center}
\begin{tikzpicture}
  \node (a) at (2,2) {$2$};
  \node (b) at (-2,0) {$4$};
  \node (c) at (2,0) {$3$};
  \node (d) at (0,-2) {$1$};
  \draw (d) -- (b)
(d) -- (c) -- (a);
\end{tikzpicture}
\end{center}

We can define a linear multi-extension of $P$ by $[7]$ by $e(1)=\{1\}$, $e(3)=\{3, 5\} $
$e(4)=\{2, 4, 6\}$, and $e(2)=\{7\} $. Then, for example, $e^{-1}(5)$ is the element of $P$ labeled by 3. This linear multi-extension contributes the $\mathfrak{m}$-permutation $1434342$ to $\tilde{\mathcal{J}}_7(P,\theta)$.
\end{example}

The following result is proven in \cite{lam2007combinatorial} by giving an explicit weight-preserving bijection between $\tilde{\mathcal{A}}(P,\theta)$ and the set of pairs
$(w, \sigma ')$ where $w \in \tilde{\mathcal{J}}_N(P,\theta)$ and $\sigma ' \in \tilde{\mathcal{A}}(C,w)$, where $C=(c_1<c_2<\ldots <c_r)$ is a chain with $r$ elements.
One can easily recover this bijection from the bijection given in the proof of Theorem~\ref{expandKhat} by restricting to $\tilde{\mathcal{A}}(P,\theta)$.

\begin{theorem}\cite[Theorem 5.6]{lam2007combinatorial}
\label{multiexpandK}
 We can write $$\tilde K_{(P,\theta)}=\displaystyle\sum_{N\geq n}\sum_{w\in \tilde{\mathcal{J}}_N(P,\theta)}\tilde L_{\mathcal{C}(w)}.$$
\end{theorem}

We now describe how to express $\tilde L_\alpha$ as an infinite linear combination of fundamental quasisymmetric functions, $\{L_\alpha\}$. Let $L_\alpha^{(i)}$ denote the 
homogeneous component of $\tilde L_\alpha$ of degree $|\alpha|+i$. 

Given $D\subset [n-1]$ and $E\subset [n+i-1]$, an injective, order-preserving map $t:[n-1]\to [n+i-1]$ is an \textit{i-extension} of 
$D$ to $E$ if $t(D)\subset E$ and $(E \backslash t(D))=([n+i-1]) \backslash t([n-1])$. In other words, $E$ is the union of the image of $D$ and
the elements not in the image of $t$. Thus $|E|=|D|+i$. Let $T(D,E)$
denote the set of $i$-extensions from $D$ to $E$. For example, if $D=\{1,2\}\subset[2]$ and $E=\{1,2,3\}\subset[3]$, then $|T(D,E)|=3$. On the other hand,
if we have $D'=\{1,2\}\subset [3]$ and $E'=\{1,3,4\}\subset[4]$, then $|T(D',E')|=0$. The proof of the following theorem is similar to that of Theorem \ref{expandLh}.

\begin{theorem}\cite[Theorem 5.12]{lam2007combinatorial} Let $\alpha$ be a composition of $n$ and $D=Des(\alpha)$ be the corresponding descent set. Then for each $i\geq0$, we have
 $$L_\alpha^{(i)} = \displaystyle\sum_{E\subset[n+i-1]}|T(D,E)|L_{\mathcal{C}(E)}.$$
\end{theorem}

\subsection{Hopf structure}
Next we describe the bialgebra structure of $\mQSym$ using the continuous basis of multi-fundamental quasisymmetric functions. The first step is to 
define the \textit{multishuffle} of two words in a fixed alphabet. To that end, we give the following definition.

\begin{definition}\label{def:multiword}
 Let $a=a_1a_2\cdots a_k$ be a word. We call $w=w_1w_2\cdots w_r$ a 
multiword of $a$ if there exists a non-decreasing, surjective
map $t:[r] \to [k]$ such that $w_j=a_{t(j)}$.
\end{definition}

As an example, consider the permutation $1342$ as a word in $\mathbb{N}$. Then $11333422$ and $1342$ are both multiwords of $1342$, while
$34442$ and $1133244$ are not multiwords of $1342$.

\begin{definition}\label{multishuffledef}
Let $a=a_1a_2\cdots a_k$ and $b=b_1b_2\cdots b_n$ be words with distinct letters.
We say that $w=w_1w_2\cdots w_m$ is a multishuffle of $a$ and $b$ if the following conditions are satisfied:
\begin{enumerate}
\item $w_i \neq w_{i+1}$ for all $i$ 
\item when restricted to $\{a_i\}$, $w$ is a multiword of $a$
\item when restricted to $\{b_j\}$, $w$ is a multiword of $b$.
\end{enumerate}
\end{definition}

Eventually we would like to multishuffle two permutations, which will not have distinct letters. We adjust our definition as follows. Given a permutation $w=w_1w_2\cdots w_k$, define $w[n]=(w_1+n)(w_2+n)\cdots(w_k+n)$ to be the word obtained by adding $n$ to each digit 
entry of $w$. For example, for $w=21$, $w[4]=65$. We then define the multishuffle of two permutations $u\in\mathfrak{S}_n$ and $w$ by
declaring it to be the multishuffle of $u$ and $w[n]$.

Starting with permutations $u=1342$ and $w=21$, we see that $v=16161346252$ is a 
multishuffle of $u=1342$ and $w[4]=65$. We shift $w$ by 4 since 4 is the largest letter in $u$.
If we restrict to the letters in $u$, $v|_u=1113422$ is a multiword of $u$, and similarly
$v|_w[4]=6665$ is a multiword of $w[4]$.

\begin{proposition}\cite[Proposition 5.9]{lam2007combinatorial}
 Let $\alpha$ be a composition of $n$ and $\beta$ be a composition of $m$. Then $$\tilde L_\alpha \tilde L_\beta = \displaystyle \sum_{u \in \rm Sh^{\m}(w_\alpha ,w_\beta[n])} \tilde L_{\mathcal{C}(u)},$$
where the sum is over all multishuffles of $w_\alpha$ and $w_\beta[n]$.
\end{proposition}

Note that this is an infinite sum whose lowest degree terms are exactly those of $L_\alpha L_\beta$, the product of the two corresponding fundamental quasisymmetric functions.

To define the coproduct, we need the following definition. 

\begin{definition}\label{def:cuut}
Let $w=w_1w_2\cdots w_k$ be a permutation. Then Cuut$(w)$ is the set of terms of the form $w_1w_2\cdots w_i \otimes w_{i+1}w_{i+2}\cdots w_k$ for $i\in[0,k]$ or of 
the form $w_1w_2\cdots w_i\otimes w_i w_{i+1}\cdots w_k$ for $i\in [1,k]$.
\end{definition}

For example, Cuut$(132)=\{\emptyset\otimes 132, 1\otimes 132, 1\otimes 32, 13\otimes 32, 13\otimes 2, 132\otimes 2, 132\otimes \emptyset \}$.
Notice how this compares to the cut coproduct of the shuffle algebra described in Section~\ref{sec:Hopf} to understand the strange spelling.

\begin{proposition}\cite[Proposition 5.10]{lam2007combinatorial} We have that
 $$\Delta(\tilde L_\alpha)=\tilde L_\alpha (x,y) = \displaystyle\sum_{u\otimes u' \in {\rm Cuut}(w_\alpha)} \tilde 
 L_{\mathcal{C}(u)}(x)\otimes \tilde L_{\mathcal{C}(u')}(y).$$
\end{proposition}

\begin{example}
 Let $\alpha=(1)$ and $\beta=(2,1)$ with $w_\alpha=1$ and $w_\beta=231$. Then
$$\tilde L_\alpha \tilde L_\beta = \tilde L_{(3,1)}+\tilde L_{(1,2,1)}+\tilde L_{(2,2)}+\tilde L_{(2,1,1)}+\tilde L_{(3,1,1)}+\tilde L_{(2,2,1,1)}+\tilde L_{(2,2,1,2)}+\ldots,$$
where the terms listed correspond to the multishuffles $1342$, $3142$, $3412$, $3421$, $13421$, $131421$, and $3414212$ of $w_\alpha$ and $w_\beta[1]$. 
We also compute 
\begin{eqnarray*}
\Delta(\tilde L_\beta)&=&\emptyset \otimes \tilde L_{(2,1)}+\tilde L_{(1)}\otimes \tilde L_{(2,1)}+\tilde L_{(1)}\otimes \tilde L_{(1,1)}+\tilde L_{(2)}\otimes \tilde L_{(1,1)} \\
& & +\tilde L_{(2)}\otimes \tilde L_{(1)}
+\tilde L_{(2,1)}\otimes \tilde L_{(1)}+\tilde L_{(2,1)}\otimes \emptyset.
\end{eqnarray*}
\end{example}

We give a combinatorial formula for the antipode map in $\mQSym$ in Theorem \ref{thm:mQSymAntipode}. In Section \ref{sec:Lhat}, we 
give an antipode map in terms of a new basis introduced within the section.

\section{The Hopf algebra of Multi-noncommutative symmetric functions}\label{MNSym}

The Hopf algebra of noncommutative symmetric functions ($\NSym$) is dual to that of quasisymmetric functions. 
We next describe a $K$-theoretic analogue called the Multi-noncommutative symmetric functions or $\MNSym$. As with $\QSym$, we recommend first being familiar with $\NSym$ as recommend \cite{stanley1999enumerative} as a reference. We recall 
the bialgebra structure of $\MNSym$ as given in \cite{lam2007combinatorial} and develop a combinatorial formula for its antipode map.

\subsection{Multi-noncommutative ribbon functions and bialgebra structure}
$\MNSym$ has a basis $\{\tilde R_\alpha\}$ of Multi-noncommutative ribbon
functions indexed by compositions, which is an analogue to the basis of noncommutative ribbon functions $\{R_\alpha\}$ for $\NSym$. 

A \textit{ribbon diagram} is a connected skew shape $\lambda/\mu$ that contains no $2\times 2$ square. There is an easy bijection between compositions and ribbon diagrams, where a ribbon diagram corresponds to the composition obtained by reading the sizes of its rows from bottom to top. See Example~\ref{ex:multiplyRs}.
It will be useful to think of $\{\tilde R_\alpha\}$ as being indexed by ribbon diagrams using this correspondence. 

We first introduce a product structure on $\{\tilde{R_\alpha}\}$
as given in \cite{lam2007combinatorial}.

\begin{proposition}\cite[Proposition 8.1]{lam2007combinatorial}
 Let $\alpha=(\alpha_1,\ldots ,\alpha_k)$ and $\beta=(\beta_1,\ldots ,\beta_m)$ be compositions. Then 
$$\tilde R_\alpha \bullet \tilde R_\beta=\tilde R_{\alpha \lhd \beta}+\tilde R_{\alpha \cdot \beta} + \tilde R_{\alpha \rhd \beta},$$ where
$\alpha \lhd \beta = (\alpha_1,\ldots , \alpha_k,\beta_1,\ldots , \beta_m)$, $\alpha \cdot \beta = (\alpha_1,\ldots,\alpha_{k-1},\alpha_k+\beta_1-1,\beta_2,\ldots,\beta_m)$, 
and $\alpha \rhd \beta =(\alpha_1,\ldots, \alpha_k+\beta_1,\beta_2, \ldots ,\beta_m)$.
\end{proposition}

\begin{example}\label{ex:multiplyRs}
 It is helpful to think of the product using ribbon diagrams. From the statement above, we have
$$ \tilde R_{(2,2)}\bullet\tilde R_{(1,2)}=\tilde R_{(2,2,1,2)}+\tilde R_{(2,2,2)}+\tilde R_{(2,3,2)}.$$ In pictures, this is shown in Figure~\ref{fig:Rmult}.
\begin{figure}[h!]
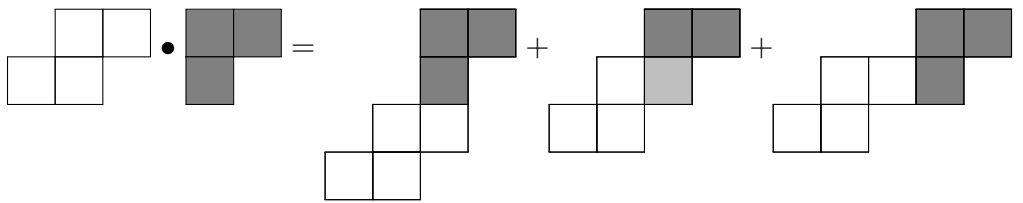

\begin{center}
\ydiagram{1+2,2} $\text{  }\bullet$ \ydiagram[*(gray)]{2, 1} \text{  }= \ydiagram{2+2,2+1,1+2,2}*[*(gray)]{2+2, 2+1}  *{2+2,2+1,1+2,2}
$\text{  }+$\ydiagram{2+2,1+2,2}*[*(gray)]{2+2}*[*(lightgray)]{2+2,2+1}*{2+2,1+2,2}
$\text{  }+$\ydiagram{3+2,1+3,2}*[*(gray)]{3+2,3+1}*{3+2,1+3,2}\text{  }.
\end{center}
\caption{Multiplying $\tilde R_{(2,2)}$ and $\tilde R_{(1,2)}$}
\label{fig:Rmult}
\end{figure}
\end{example}

In contrast to the product in $\mQSym$, the product in $\MNSym$ is a finite sum whose highest degree terms are those of the corresponding 
product $R_\alpha R_\beta$ in $\NSym$.

\begin{proposition} The coproduct of a basis element is
 $$\Delta(\tilde R_\alpha)=\sum_{w_\alpha \in {\rm Sh}^\m (w_\beta,w_\delta[i])}{\tilde R_\beta \otimes \tilde R_\delta},$$ where $i \in \mathbb{N}$ and 
$w_\beta\in \mathfrak{S}_i$.
\end{proposition}

Note that since multishuffles of $w_\beta$ and $w_\delta[i]$ may not have adjacent letters that are equal, we may define the descent set of a multishuffle of 
$w_\beta$ and $w_\delta[i]$ in the usual way.

\begin{example}\label{MNSymcoproductexample}
 In general, computing the coproduct in $\MNSym$ is not an easy task. However, for compositions with only one part, we have
$$\Delta(\tilde R_{(n)})=\tilde R_{(n)}\otimes 1 + \tilde R_{(n-1)}\otimes \tilde R_{(1)}+\tilde R_{(n-2)}\otimes \tilde R_{(2)} + \ldots \tilde R_{(1)}\otimes \tilde R_{(n-1)}+ 1 \otimes \tilde R_{(n)}$$
because the only way that a multishuffle of two permutations results in an increasing sequence is for it to be the concatenation of two increasing permutations. 
We use this fact in the proof of the antipode in $\MNSym$.
\end{example}

\subsection{Antipode map for $\MNSym$}

Suppose we have a ribbon shape corresponding to $\alpha$, a composition of $n$. We say that ribbon shape $\beta$
is a \textit{merging} of ribbon shape $\alpha$ if we can obtain shape $\beta$ from shape $\alpha$ by merging pairs of boxes that share an edge. 
The order in which the pairs are merged does not matter, only set of boxes that were merged. Let $M_{\alpha,\beta}$
be the number of ways to obtain shape $\beta$ from shape $\alpha$ by merging. We will label each box in the ribbon shape 
to keep track of our actions. 

\begin{example}
Let $\alpha = (2,2,1)$ and $\beta=(2,1)$. Then $M_{\alpha,\beta}=3$. The labeled ribbon shape $\alpha$ and the three mergings 
resulting in shape $\beta$ are shown in Figure~\ref{fig:mergings}.

\ytableausetup{boxsize=.55cm}
\begin{figure}[h!]
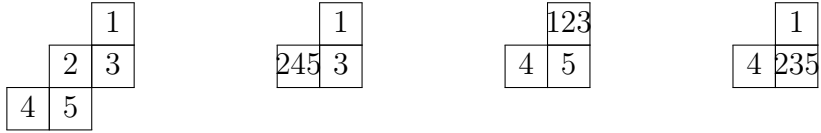

\begin{center}
 $\text{      }$
\begin{ytableau}
\none & \none & 1  \\
\none & 2 & 3 \\
4 & 5 
\end{ytableau}\hspace{.7in}
\begin{ytableau}
 \none & 1 \\
245 & 3
\end{ytableau}\hspace{.7in}
\begin{ytableau}
 \none & 123 \\
4 & 5
\end{ytableau}\hspace{.7in}
\begin{ytableau}
 \none & 1 \\
4 & 235
\end{ytableau}
\end{center}
\caption{Ribbon shape $(2,2,1)$ and its three mergings of ribbon shape $(2,1)$}
\end{figure}
\label{fig:mergings}
\end{example}

In the following sections, $\omega$ will denote the fundamental involution of the symmetric functions defined by $\omega(e_n)=h_n$ for elementary symmetric function $e_n$, complete homogeneous symmetric function $h_n$, and for all $n$. 

If $\alpha=(\alpha_1,\ldots,\alpha_k)$ is a composition of $n$, define $rev(\alpha)=(\alpha_k,\ldots,\alpha_1)$. Now let $\omega(\alpha)$ be the unique composition of $n$ whose partial sums $S_{\omega(\alpha)}$ form the complementary set within $[n-1]$ to the partial sums $S_{rev(\alpha)}$. Alternatively, we may think of the ribbon shape $\alpha$ as $\lambda / \mu$ for Young diagrams $\lambda$ and $\mu$. Then $\omega(\alpha)$ is the ribbon shape $\lambda^t/ \mu^t$, where $\lambda^t$ and $\mu^t$ are the transposes of $\lambda$ and $\mu$, respectively. The number of blocks in each row of $\omega(\alpha)$ reading from bottom to top corresponds to the number of blocks in each column of $\alpha$ reading from right to left. For example, if $\alpha = (2,1,1,3)$, $\omega(\alpha)=(1,1,4,1)$.

\begin{theorem}\label{MNSymAntipode}
 Let $\alpha$ be a composition of $n$. Then $$S(\tilde R_\alpha) = (-1)^n \displaystyle\sum_\beta M_{\omega(\alpha),\beta} \tilde R_\beta,$$ where we sum of over all compositions $\beta$.
\end{theorem}

Note that only finitely many terms will be nonzero because $M_{\omega(\alpha),\beta}=0$ if $|\beta|>|\alpha|$.

\begin{proof}
 We prove this by induction on the number of parts of the composition $\alpha$. 

We compute directly that $$0=S(\tilde R_{(1)})=S(\tilde R_{(1)})\cdot 1 + S(1)\cdot \tilde R_{(1)}=S(\tilde R_{(1)})+\tilde R_{(1)}$$ by Proposition \ref{antiendomorphism}, 
so $S(\tilde R_{(1)})=-\tilde R_{(1)}$.

Now assume that $S(\tilde R_{(k)})=(-1)^k \displaystyle\sum_{i=0}^{k-1} \binom{k-1}{i} \tilde R_{1^{i+1}}$ for all $k<n$. 
Then, using Example \ref{MNSymcoproductexample} and Definition \ref{convolutionalinverse}, we see that
\begin{eqnarray}
0 &=& \tilde R_{(n)}+S(\tilde R_{(n)}) + \displaystyle\sum_{i=1}^{n-1} S(\tilde R_{(i)})\bullet \tilde R_{(n-i)} \nonumber \\
&=& \tilde R_{(n)}+S(\tilde R_{(n)})+ \displaystyle \sum_{i=1}^{n-1}\left((-1)^i \displaystyle\sum_{j=0}^{i-1} \binom{i-1}{j} \tilde R_{(1^{j+1})}\right)\bullet \tilde R_{(n-i)} \nonumber \\
&=& \tilde R_{(n)}+S(\tilde R_{(n)})+ \displaystyle \sum_{i=1}^{n-1}(-1)^i \displaystyle\sum_{j=0}^{i-1} \binom{i-1}{j} (\tilde R_{(1^{j+1},n-i)}+\tilde R_{(1^j,n-i+1)}+\tilde R_{(1^j, n-i)}). \nonumber 
\end{eqnarray}

There are five types of terms that appear in this sum.
\begin{enumerate}
 \item $\tilde R_{(1^{n-m},m)}$. The coefficient of this term is 
 
 $$(-1)^{n-m}\binom{n-m-1}{n-m-1}+(-1)^{n-m+1}\binom{n-m}{n-m}=0.$$
 
\item $\tilde R_{(m)}$, where $1<m<n$. The coefficient of this term is 

$$(-1)^{n-m+1}\binom{n-m}{0}+(-1)^{n-m}\binom{n-m-1}{0}=0.$$

\item $\tilde R_{(1^s,m)}$, where $s<n-m$, and $m>1$. The coefficient of this term 

$$(-1)^{n-m}\binom{n-m-1}{s-1}+(-1)^{n-m+1}\binom{n-m}{s}+(-1)^{n-m}\binom{n-m-1}{s}=0.$$

\item $\tilde R_{(1^k)}$, where $k\leq n$. The coefficient of this term is $$(-1)^{n-1}\binom{n-2}{k-2}+(-1)^{n-1}\binom{n-2}{k-1} = (-1)^{n-1}\binom{n-1}{k-1}.$$
\item $\tilde R_{(n)}$. The coefficient of this term is $$(-1)^1\binom{0}{0}=-1.$$
\end{enumerate}

Thus 
\[
 0=S(\tilde R_{(n)})+(-1)^{n-1}\displaystyle\sum_{s=1}^n\binom{n-1}{s-1} \tilde R_{1^s},
\]
and so 
\[
 S(\tilde R_{(n)})=(-1)^n\displaystyle\sum_{s=0}^{n-1}\binom{n-1}{s}\tilde R_{(1^{s+1})}. 
\]

It is clear that there are $\binom{n-1}{s}$ mergings of $\omega(\alpha)=(1^n)$ that result in shape $(1^{s+1})$ since we are choosing $s$ of the $n-1$ border edges to remain intact.

Now suppose $S(\tilde R_\alpha) = (-1)^n \displaystyle\sum_\beta M_{\omega(\alpha),\beta} \tilde R_\beta$ holds for all compositions $\alpha$ with up to $k-1$ parts, and let $\beta=(\beta_1,\beta_2,\ldots,\beta_k)$
be a composition with $k$ parts. We know that 

\[
\tilde R_{\beta}=\tilde R_{(\beta_1,\beta_2,\ldots,\beta_{k-2},\beta_{k-1})}\bullet \tilde R_{(\beta_k)}-\tilde R_{(\beta_1,\beta_2,\ldots,\beta_{k-2},\beta_{k-1}+\beta_k)}-R_{(\beta_1,\beta_2,\ldots,\beta_{k-2},\beta_{k-1}+\beta_k-1)},
\]

and so 
\begin{eqnarray*}
S(\tilde R_{\beta})&=&S(\tilde R_{(\beta_k)})\bullet S(\tilde R_{(\beta_1,\beta_2,\ldots,\beta_{k-2},\beta_{k-1})}) -S(\tilde R_{(\beta_1,\beta_2,\ldots,\beta_{k-2},\beta_{k-1}+\beta_k)})\\
& & -S(\tilde R_{(\beta_1,\beta_2,\ldots,\beta_{k-2},\beta_{k-1}+\beta_k-1)}).
 \end{eqnarray*}

In Figure~\ref{fig:NSymantipodeschematic}, let the thin rectangle represent all mergings of $\omega(\beta_k)$ and the square represent all mergings of $\omega(\beta_1,\ldots,\beta_{k-1})$.

\begin{figure}[h!]
\begin{center}
\includegraphics[width=4.5in]{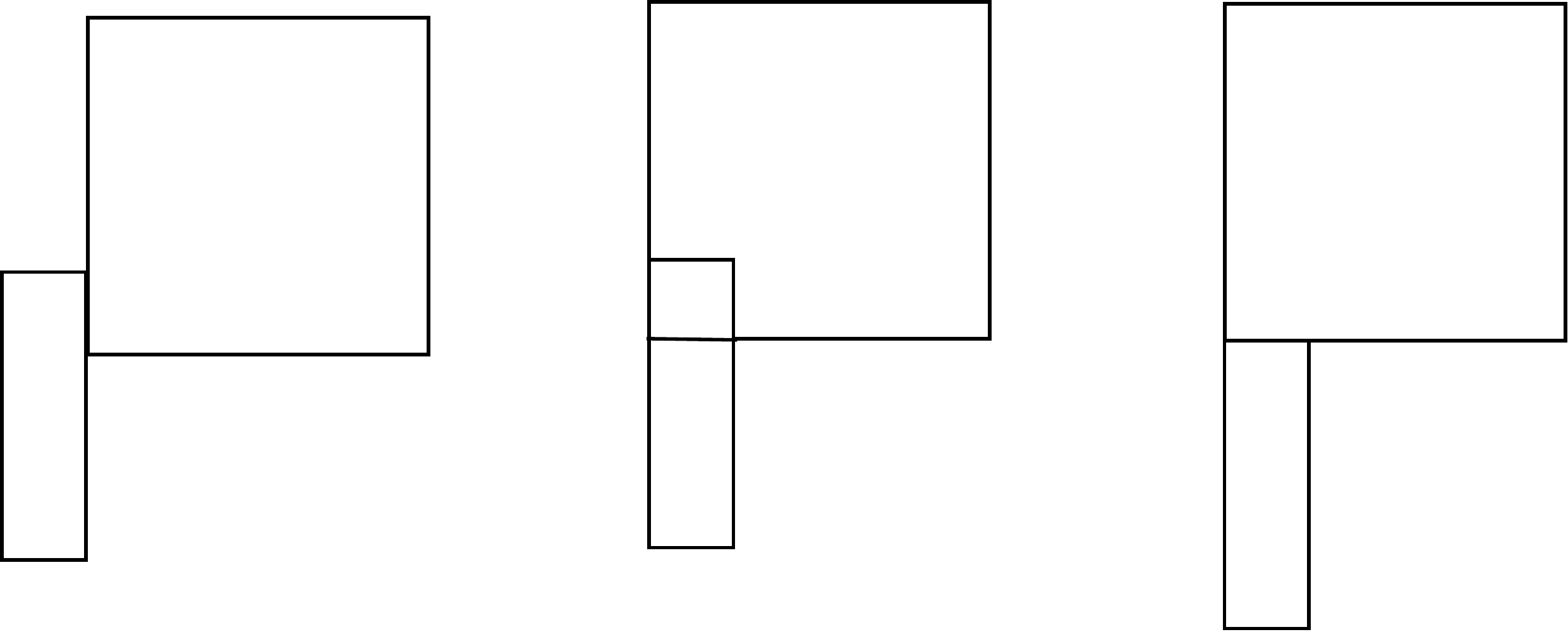}
\end{center}
\begin{center}
(1) \hspace{1.4in} (2) \hspace {1.3in} (3)
\end{center}
\caption{$\NSym$ antipode merging schematic}
\label{fig:NSymantipodeschematic}
\end{figure}

Then the image labeled (1) represents all mergings obtained by adding the last part of a merging of $\omega(\beta_k)$ 
to the first part of a merging of $\omega(\beta_1,\ldots, \beta_{k-1})$. The image labeled (2) represents all mergings obtained by merging the topmost box
in a merging of $\omega(\beta_k)$ with the bottom leftmost box of a merging of $\omega(\beta_1,\ldots,\beta_k)$. These two mergings with multiplicities are exactly the shapes 
we want in $S(\tilde R_{\beta})$.

The image labeled (3) represents all mergings obtained by concatenating a merging of $\omega(\beta_k)$ with a merging of $\omega(\beta_1,\ldots,\beta_{k-1})$.
We do not want these mergings to appear in $S(\tilde R_\beta)$ because it is impossible for boxes that are side-by-side in $\omega(\beta)$ to be stacked one on top of the other
in a merging of $\omega(\beta)$.

We use the fact that $S(\tilde R_{\beta_k})\bullet S(\tilde R_{(\beta_1,\ldots,\beta_{(k-1)})})$ results in all mergings of type (1), (2), and (3), 
$S(\tilde R_{(\beta_1,\ldots, \beta_{k-1}+\beta_k)})$ gives all mergings of type (2) and (3), and $S(\tilde R_{(\beta_1,\ldots, \beta_{k-1}+\beta_k-1)})$
contains exactly those mergings of type (2). The parity of the sizes of the compositions provides the necessary cancellation and
leaves us with all mergings of type (1) and (2), as desired. 
\end{proof}

\begin{example}
Consider $S(\tilde R_{(1,2)})=S(\tilde R_{(1)})\bullet S(\tilde R_{(1,1)})-S(\tilde R_{(1,1,1)})-S(\tilde R_{(1,1)})$. The image below shows all of the mergings in $S(\tilde R_{(1)})\bullet S(\tilde R_{(1,1)})$ in the first line with the proper sign, subtracts mergings of $S(\tilde R_{(1,1,1)})$
in the second line, and subtracts mergings of $S(\tilde R_{(1,1)})$ in the third line. The black boxes represent mergings of $\omega(1)=(1)$, the white boxes represent mergings of $\omega(1,1)=(2)$, and the gray boxes represent boxes where the two shapes have merged.
 \begin{center}
$ - \begin{ytableau} *(black) & & \end{ytableau}-
 \begin{ytableau} $ $ & \\ *(black) \end{ytableau}-
 \begin{ytableau}*(gray) & \end{ytableau}-
 \begin{ytableau}*(black) & \end{ytableau}-
 \begin{ytableau} $ $ \\*(black)\end{ytableau} - \begin{ytableau}*(gray) \end{ytableau}$ \\
 \vspace{.1in}
 $+\begin{ytableau} *(black) & & \end{ytableau}+\begin{ytableau} *(black) & \end{ytableau} + \begin{ytableau} *(gray) & \end{ytableau} + \begin{ytableau} *(gray) \end{ytableau}$\\
  \vspace{.1in}
 $-\begin{ytableau}*(gray) & \end{ytableau} - \begin{ytableau} *(gray) \end{ytableau}$\\
  \vspace{.1in}
 $=- \begin{ytableau} $ $ & \\ *(black) \end{ytableau}- \begin{ytableau} $ $ \\ *(black)\end{ytableau}-\begin{ytableau} *(gray) & \end{ytableau} - \begin{ytableau} *(gray)\end{ytableau}$
 \end{center}

\end{example}

\subsection{Antipode map for $\mQSym$}

We know from \cite[Theorem 8.4]{lam2007combinatorial} that the bases $\{\tilde L_\alpha \}$ and $\{\tilde R_\alpha\}$ satisfy the criteria in 
Lemma \ref{lem:pairingantipode}. Extending the definition below by continuity gives the following antipode formula in $\mQSym$.

\begin{theorem}\label{thm:mQSymAntipode} Let $\alpha$ be a composition of $n$. Then 
 $$S(\tilde L_\alpha) = \displaystyle\sum_\beta (-1)^{|\beta|} M_{\beta,\omega(\alpha)}\tilde L_{\beta},$$ where the sum is over all compositions $\beta$.
\end{theorem}

Note that while $S(\tilde R_\alpha)$ is a finite sum of Multi-noncommutative ribbon functions for any $\alpha$, 
$S(\tilde L_\alpha)$ is an infinite sum of multi-fundamental quasisymmetric functions for any $\alpha$. Since any arbitrary linear combination 
of multi-fundamental quasisymmetric functions is in $\mQSym$, this is an admissible antipode formula.

\section{A new basis for $\mQSym$}\label{sec:Lhat}
\subsection{$(P,\theta)$-multiset-valued partitions}
To create a new basis for $\mQSym$, which will be useful in finding antipode formulas, 
we extend the definition of a $(P,\theta)$-set-valued partition to what we call a $(P,\theta)$\textit{-multiset-valued partition} in the natural way.
In a $(P,\theta)$-multiset-valued partition $\sigma$, we allow $\sigma(p)$ for $p\in P$ to be a finite multiset of positive integers, keeping all other
definitions the same. An example of a $(P,\theta)$-multiset-valued partition is shown in Figure~\ref{fig:mvaluedpartition}. 

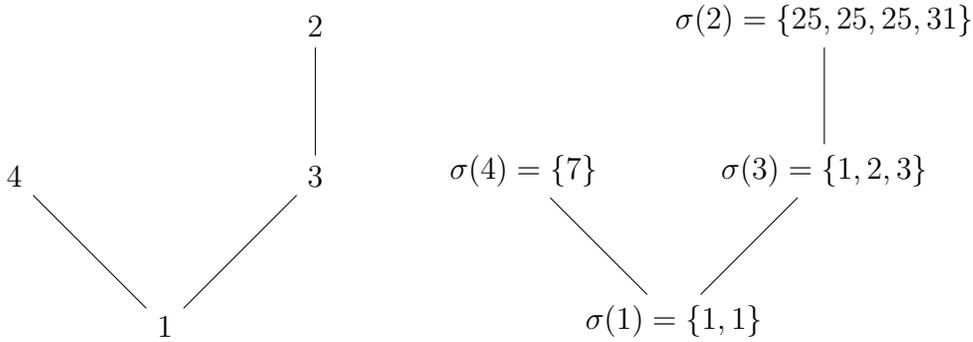
\begin{figure}[h!]
\begin{center}
\begin{tikzpicture}
  \node (a) at (2,2) {$2$};
  \node (b) at (-2,0) {$4$};
  \node (c) at (2,0) {$3$};
  \node (d) at (0,-2) {$1$};
  \draw (d) -- (b)
(d) -- (c) -- (a);
\end{tikzpicture}\hspace{.5in}
\begin{tikzpicture}
  \node (a) at (2,2) {$\sigma(2)=\{25,25,25,31\}$};
  \node (b) at (-2,0) {$\sigma(4)=\{7\}$};
  \node (c) at (2,0) {$\sigma(3)=\{1,2,3\}$};
  \node (d) at (0,-2) {$\sigma(1)=\{1,1\}$};
  \draw (d) -- (b)
(d) -- (c) -- (a);
\end{tikzpicture}
\end{center}
\caption{A $(P,\theta)$-multiset-valued partition}
\label{fig:mvaluedpartition}
\end{figure}

Now define $\tA(P,\theta)$ to be the set of all $(P,\theta)$-multiset-valued partitions. For each positive integer $i$, let 
$\sigma^{-1}(i)$ be the multiset $\{x \in P \mid i=\sigma(x)\}$. In other words, $\sigma^{-1}(i)$ lists $p\in P$ once for each occurence of $i$ in $\sigma(p)$. In the example in Figure~\ref{fig:mvaluedpartition}, $\sigma^{-1}(1)=\{1,1,3\}$. Now 
define $\tK _{P,\theta}\in\mathbb{Z}[[x_1,x_2,\ldots ]]$ by $$\tK_{P,\theta}=\displaystyle
\sum_{\sigma\in \tA(P,\theta)}{x_1^{\#\sigma^{-1}(1)}x_2^{\#\sigma^{-1}(2)}\cdots}.$$

Using this multiset analogue of our definitions, we define $$\L_\alpha=\tK_{(P,w)}=\displaystyle
\sum_{\sigma\in \tA(P,w)}{x_1^{\#\sigma^{-1}(1)}x_2^{\#\sigma^{-1}(2)}\cdots},$$ 
where $P=p_1<\ldots<p_k$ is a finite linear order and $w\in \mathfrak{S}_k$ with $\mathcal{C}(w)=\alpha$.

\subsection{Properties}

Recall the definition of $T(D,E)$ from Section \ref{mQSym}. For the proof of the following theorem, it may be useful to review the $(P,w)$-partition definition of the fundamental quasisymmetric function $L_\alpha$: 

\[L_\alpha=\displaystyle
\sum_{\sigma\in \mathcal{A}(P,w)}{x_1^{\#\sigma^{-1}(1)}x_2^{\#\sigma^{-1}(2)}\cdots},
\]
where $\mathcal{C}(w)=\alpha$ and $\mathcal{A}(P,w)$ is the set of all $(P,w)$-partitions.
We refer the reader to \cite{lam2007combinatorial} for further reading as we use the same notation. The following proof is almost identical to the proof of the analogous result for $\tilde{L}_\alpha$ given for Theorem 5.12 in \cite{lam2007combinatorial}. 

\begin{theorem}\label{expandLh} Let $\alpha$ be a composition of $n$ and $D(\omega(\alpha))=Des(\omega(\alpha))$.
We have that 
$$\L_\alpha ^{(i)} = \displaystyle\sum_{E \subset [n+i - 1]} |T(D(\omega(\alpha)),E)| L_{\omega(\mathcal{C}(E))}.$$
\end{theorem}
\begin{proof}
 Let $w=w_\alpha$ and consider the subset $\tA_i(C,w)\subset \tA(C,w)$ consisting of multiset-valued $(C,w)$-partitions $\sigma$ of size
$|\sigma|=n+i$, where $C=(c_1<c_2<\cdots<c_n)$ is a chain. We must show that the generating function of $\tA_i(C,w)$ is equal to $$\displaystyle\sum_{E \subset [n+i - 1]} |T(D(\omega(\alpha)),E)| L_{\omega(\mathcal{C}(E))}.$$  Indeed, for each pair
$t\in T(D(\omega(\alpha)),E)$ for some $E$, the function $L_{\omega(\mathcal{C}(E))}$ is the generating function  of all $\sigma\in \tA_i(C,w)$ satisfying $|\sigma(c_j)|=t(n-(j-1))-t(n-j))$, where
$t(0)=0$ and $t(n)=n+i$. Letting $C'=c'_1<c'_2,\ldots , c'_{n+i}$ be a chain with $n+i$ elements, we obtain a $(\mathcal{C}', \omega(\mathcal{C}(E)))$-partition 
$\sigma ' \in \mathcal{A}(C',\omega(\mathcal{C}(E)))$ by assigning the elements of $\sigma(c_i)$ in increasing order to
$c'_{t(i-1)+1},\ldots ,c'_{t(i)}$.
\end{proof}

\begin{example}
 Letting $\alpha =(1,2,1,2)$, we see that $$\L _\alpha^{(1)} = L_{(2,2,1,2)}+2L_{(1,3,1,2)}+L_{(1,2,2,2)}+2L_{(1,2,1,3)}.$$
In this example, the coefficient of $L_{(1,3,1,2)}$ is $2$ because 
\begin{eqnarray*}
|T(D(\omega(1,2,1,2)), \{1,4,5\}\subset [n+1-1=6])|&=&|T(D(1,3,2),\{1,4,5\})| \\ 
&=&|T(\{2,4\},\{1,4,5\})| \\
&=&2.
\end{eqnarray*}
\end{example}

Given the basis of multi-quasisymmetric functions, $\{\tilde L_\alpha\}$, the set $\{\L_\alpha\}$ is natural to consider because of the next proposition. We remind the reader that $\omega(L_\alpha)=L_{\omega(\alpha)}$ in $\QSym$.

\begin{proposition}\label{Lhatisabasis}
 We have $\omega(\tilde L_\alpha) = \L_{\omega(\alpha)}$, and the set of $\{\L_\alpha\}$ forms a continuous basis for $\mQSym$.
\end{proposition}
 
\begin{proof}
 Using Proposition \ref{expandLh}, $$\omega(\tilde L_\alpha)=\omega\left(\displaystyle\sum_{E \subset [n+i - 1]} |T(D(\alpha),E)| L_{\mathcal{C}(E)}\right)=
\displaystyle\sum_{E \subset [n+i - 1]} |T(D(\alpha),E)| L_{\omega(\mathcal{C}(E))}=\L_{\omega(\alpha)}.$$
\end{proof}

We have an analogue of Stanley's Fundamental Theorem of P-partitions for our new basis of $\L_\alpha$'s. The proof of this result follows closely that of
Theorem \ref{multiexpandK} in this paper given in \cite{lam2007combinatorial}.

\begin{theorem}\label{expandKhat} 
Suppose poset $P$ has $n$ elements. We have $$\K_{P,\theta} = \sum_{N \geq n}
\sum_{w \in \tJ_{N(P,\theta)}} {\L_{\mathcal{C}(w)}}.$$
\end{theorem}

\begin{proof}
We prove this result by giving an explicit weight-preserving bijection between
$\tA(P,\theta)$ and the set of pairs $(w,\sigma')$ where $ w\in
\tJ_N(P,\theta)$ and $\sigma' \in \tA(C,w)$ where $C = (c_1 <
c_2 < \cdots < c_l)$ is a chain with $l = \ell(w)$ elements.  Let
$\sigma \in \tA(P,\theta)$. For each $i$, let $\sigma^{-1}(i)$ denote
the submultiset of $[n]$ via $\theta$, and let $w^{(i)}_\sigma$ denote
the word of length $|\sigma^{-1}(i)|$ obtained by writing the
elements of $\sigma^{-1}(i)$ in increasing order.  Note that it is 
possible for $w^{(i)}_j=w^{(i)}_{j+1}$. This will occur when the 
letter $i$ appears more than once in some $\sigma(s)$ for $s\in P$. 

Let $w$ denote the
unique $\m$-permutation such that $w_\sigma:=w_\sigma^{(1)}
w_\sigma^{(2)} \cdots$ is a multiword of $w$ and $t:
\ell(w_\sigma) \to \ell(w)$ be the associated function as in
Definition~\ref{def:multiword}.  We know that $w_\sigma$ is a finite word because $\sigma^{(-1)}(r) =
\emptyset$ for sufficiently large $r$. Note that $w_\sigma$ uses all letters $[n]$.
Now define $\sigma' \in \tA(C,w)$ by

$$\sigma'(c_i) = \{r^k \mid r\text{ is a positive integer and } w_\sigma^{(r)} \text{ contributes } k \text{ letters to } w_\sigma|_{t^{-1}(i)}\}$$
where $w_\sigma|_{t^{-1}(i)}$ is the set of letters in $w_\sigma$ at
the positions in the interval $t^{-1}(i)$.  We will show that this
defines a map $\alpha:\sigma \mapsto (w,\sigma')$ with the required
properties.

First, $w$ is the $\m$-permutation associated to the linear multi-extension
$e_w$ of $P$ by $[\ell(w)]$ defined by the condition that $e_w(x)$
contains $j$ if and only if $w_j = \theta(x)$. It follows from the
definition that this $e_w: P \to 2^{[1,\ell(w)]}$ is a linear
multi-extension. To check that $\sigma'$ is a multiset-valued $(C,w)$-partition, we note that $\sigma'(c_i)\leq \sigma'(c_{i+1})$ because
the function $t$ is non-decreasing. Moreover,
if $w_i > w_{i+1}$, then $\sigma'(c_i) < \sigma'(c_{i+1})$ because
each $w_\sigma^{(r)}$ is weakly increasing.

We define the inverse map $\beta: (w,\sigma') \mapsto \sigma$ by the formula
$$
\sigma(x) = \bigcup_{j \in e_w(x)} \sigma'(c_j).
$$
The $(P,\theta)$-multiset-valued partition $\sigma$ respects $\theta$ because
$e_w$ is a linear multi-extension. Thus if $x<y$ in $P$ and $\theta(x)>\theta(y)$, then $\sigma(x)<\sigma(y)$ since
$e_w(x)<e_w(y)$ and there is a descent in $w$ between the corresponding entries of $\theta(x)$ and $\theta(y)$.

Then
$\beta \circ \alpha = {\rm id}$ follows immediately. For $\alpha
\circ \beta = {\rm id}$, consider a submultiset $\sigma'(c_j) \subset
\sigma(x)$. One checks that this submultiset gives rise to
$|\sigma'(c_j)|$ consecutive letters all equal to $\theta(x)$ in
$w_\sigma$ and that this is a maximal set of consecutive repeated
letters. This shows that one can recover $\sigma'$.  To see that $w$
is recovered correctly, one notes that if $\sigma'(c_j)$ and
$\sigma'(c_{j+1})$ contain the same letter $r$ then $w_j < w_{j+1}$
so by definition $w_j$ is placed correctly before $w_{j+1}$ in
$w^{(r)}_\sigma$.

\end{proof}

\begin{example}\label{ex:Lhatbijection}
 Let $\theta$ be the labeling
\begin{center}
\begin{ytableau}{3}&{4}&{5} \\ {1}&{2} \end{ytableau}
\end{center}
of the shape $\lambda = (3,2)$. (Note that this is the labeling $\theta_s$ described in Section~\ref{sec:PthetaPartitions}.) Take the
$(\lambda,\theta)$-multiset-valued partition 
\begin{center}
\ytableausetup{boxsize=.7cm}
\begin{ytableau}{112}&{23}&{345} \\ {45}&{667} \end{ytableau}
\end{center}
in $\tA(\lambda,\theta)$.  Then we have $$w_{\sigma} =
(3,3;3,4;4,5;1,5;1,5;2,2;2),$$ where for example,
$w_{\sigma}^{(1)}=(3,3)$ since the cell labeled $3$ contains two copies of
the number $1$ in $\sigma$. Therefore $$w = (3,4,5,1,5,1,5,2)$$ and
the corresponding composition $\mathcal{C}(w)$ is $(3,2,2,1)$.  Then $\sigma'$
written as sequence is
$$\{1,1,2\},\{2,3\},\{3\},\{4\},\{4\},\{5\},\{5\}, \{6,6,7\}.$$ For example
$\sigma'(c_1) = \{1,1,2\}$ since $w_{\sigma}^{(1)}$ contributes two $3$'s and
$w_{\sigma}^{(2)}$ contributes one $3$ to the beginning of
$w_{\sigma}$.

To obtain the inverse map, $\beta$, read
$w$ and $\sigma'$ in parallel and place $\sigma'(c_i)$ into cell
$\theta_s^{-1}(w_i)$. For example, we put $\{1,1,2\}$ into the cell
labeled $3$, and we put $\{2,3\}$ into the cell labeled $4$.

The linear multi-extension, $e_w$ in this example can be represented by the filling below. 
\begin{center}
\begin{ytableau}{1}&{2}&{357} \\ {46}&{8} \end{ytableau}
\end{center}
\end{example}

\subsection{Antipode}

Recall that in $\QSym$, 

\[S(L_\alpha)=(-1)^{|\alpha|}L_{\omega(\alpha)} = (-1)^{|\alpha|}\omega(L_\alpha)=\omega(L_\alpha(-x_1,-x_2,\ldots)).\] 
Using the set $\{\L_\alpha\}$, we have a similar result in $\mQSym$.

\begin{theorem}\label{LhatAntipode}
 In $\mQSym$, $$S(\tilde L_\alpha)=\L_{\omega(\alpha)}(-x_1,-x_2,\ldots).$$
\end{theorem}

\begin{proof}
Using Theorem~\ref{multiexpandK} and the antipode in $\QSym$, we see that
\begin{eqnarray}
 S(\tilde L_\alpha) &=& S\left(\displaystyle\sum_{E \subset [n+i - 1]} |T(D,E)| L_{\mathcal{C}(E)}\right) \nonumber \\
&=& \displaystyle\sum_{E \subset [n+i - 1]} |T(D,E)| S(L_{\mathcal{C}(E)}) \nonumber \\
&=& \displaystyle\sum_{E \subset [n+i - 1]} |T(D,E)| (-1)^{|\mathcal{C}(E)|}L_{\omega(\mathcal{C}(E))} \nonumber \\
&=& \L_{\omega(\alpha)}(-x_1,-x_2,\ldots). \nonumber 
\end{eqnarray}
\end{proof}

\section{The Hopf algebra of multi-symmetric functions}\label{mSym}
We next describe the space of multi-symmetric functions, $\mSym$. We refer the reader to \cite{lam2007combinatorial} for details. 

\subsection{Set-valued tableaux}\label{sec:setvaluedtableaux}

Recall the definition of $\tilde{K}_{\lambda,\theta_s}$ from Section~\ref{sec:PthetaPartitions}.We define
$$\mSym=\prod_{\lambda}\mathbb{Z}\tilde  K_{\lambda,\theta_s}$$ to be the subspace of $\mQSym$ continuously spanned by the $\tilde K_{\lambda,\theta_s}$, where $\lambda$ varies
over all partitions. From this point forward, we will write $\tilde K_\lambda$ in place of $\tilde{K}_{\lambda,\theta_s}$ and call a $(\lambda/\mu,\theta_s)$-set-valued partition 
a \textit{set-valued tableau} of shape $\lambda/\mu$. We will think of these tableaux as fillings of a Young diagram with finite, nonempty subsets of positive integers such that the subsets are weakly increasing across rows and strictly increasing down columns, where subsets are ordered as in Section~\ref{sec:multifunquasis}. For any set-valued tableau $T$, $|T|$ denotes the sum of the sizes of the subsets labeling the boxes of $T$.

\begin{example}
 For $\lambda=(2,1)$, we have 
$\tilde K_{\lambda}=x_1^2x_2+ 2x_1x_2x_3+x_1^2x_2^2+3x_1^2x_2x_3+8x_1x_2x_3x_4+\ldots$, corresponding
to the following labeled poset.
\begin{center}
\begin{tikzpicture}[scale=1.5]
\node (A) at (0,0) {$2$};
\node (B) at (-1,1) {$3$};
\node (C) at (1,1) {$1$};
\draw (A) -- (B)
(A) -- (C);
\end{tikzpicture}
\end{center}
The set-valued tableaux corresponding to the first four terms are shown below. Note that we omit commas in the subsets---the box in position row one and column two of the fourth tableau is filled with $\{1,2\}$.
\begin{center}
\begin{ytableau}
1 & 1 \\ 2
\end{ytableau}\qquad
\begin{ytableau}
1 & 2 \\ 3
\end{ytableau}\qquad
\begin{ytableau}
1 & 3 \\ 2
\end{ytableau}\qquad
\begin{ytableau}
1 & 12 \\ 2
\end{ytableau}\qquad
\begin{ytableau}
1 & 12 \\ 3
\end{ytableau}\qquad
\begin{ytableau}
1 & 13 \\ 2
\end{ytableau}\qquad
\begin{ytableau}
1 & 1 \\ 23
\end{ytableau}
\end{center}
\end{example}

\subsection{Basis of stable Grothendieck polynomials}\label{sec:Grothendieckfirst}
We now introduce another (continuous) basis for $\mSym$, the \textit{stable Grothendieck polynomials}. Stable Grothendieck polynomials originated from the Grothendieck
polynomials of Lascoux and Sch\"{u}tzenberger \cite{lascoux1structure}, which served as representatives of $K$-theory classes of structure sheaves of Schubert varieties. 
Through the work of Fomin and Kirillov \cite{fomin1996yang} and Buch \cite{buch2002littlewood}, the stable Grothendieck polynomials, a limit of the Grothendieck polynomials,
were discovered and given the combinatorial interpretation in the theorem below. These symmetric functions play the role of Schur functions in the $K$-theory of Grassmannians. 

\begin{theorem}\cite[Theorem 3.1]{buch2002littlewood}\label{thm:GrothendieckDef}
 The stable Grothendieck polynomial $G_{\lambda/\mu}$ is given by the formula $$G_{\lambda/\mu}=\displaystyle\sum_T(-1)^{|T|-|\lambda/\mu|}x^T,$$
where the sum is taken over all set-valued tableaux of shape $\lambda/\mu$.
\end{theorem}

The stable Grothendieck polynomials are related to the $\tilde K_\lambda$ by 
$$\tilde K_\lambda(x_1,x_1,\ldots)=(-1)^{|\lambda|}G_{\lambda}(-x_1,-x_2,\ldots).$$

\begin{remark}\label{rem:Gamma}
In \cite{buch2002littlewood}, Buch studied a bialgebra $\Gamma=\oplus_\lambda \mathbb{Z} G_\lambda$ spanned by the set of stable Grothendieck polynomials.
Note that the bialgebra $\Gamma$ is not the same as $\mSym$. In particular, the antipode formula given in Theorem~\ref{thm:mSymAntipode} is 
valid in $\mSym$ but not in $\Gamma$ as only finite linear combinations of stable Grothendieck polynomials are allowed in $\Gamma$.
\end{remark}

\subsection{Weak set-valued tableaux}\label{sec:weaksetvalued}

The following definition is needed to introduce one final basis for $\mSym$, $\{J_\lambda\}$.

\begin{definition}
 A weak set-valued tableau T of shape $\lambda/\nu$ is a filling of the boxes of the skew shape $\lambda/\nu$ with finite, non-empty multisets of positive
integers so that 
\begin{itemize}
 \item[(1)] the largest number in each box is stricty smaller than the smallest number in the box directly to the right of it, and
\item[(2)] the largest number in each box is less than or equal to the smallest number in the box directly below it. 
\end{itemize}
\end{definition}

In other words, we fill the boxes with multisets so that rows are strictly increasing and columns are weakly increasing. For example, the filling of shape $(3,2,1)$ shown below gives
a weak set-valued tableau, $T$, of weight $x^T= x_1 x_2^3 x_3^3 x_4^2 x_5 x_6 x_7$.

\begin{center}
\begin{ytableau}
12 & 33 & 46 \\
223 & 4  \\
57 \\
\end{ytableau}
\end{center}

Let $J_{\lambda/\nu} = \displaystyle\sum_T x^T$ be the weight generating function of weak set-valued tableaux $T$ of shape $\lambda/\nu$.

\begin{theorem}\cite[Proposition 9.22]{lam2007combinatorial}
 For any skew shape $\lambda/\nu$, we have $$\omega(\tilde K_{\lambda/\nu})=J_{\lambda/\nu}.$$
\end{theorem}

\section{The Hopf algebra of Multi-symmetric functions}\label{MSym}
\subsection{Reverse plane partitions}
We next introduce the big Hopf algebra of Multi-symmetric function, $\MSym$, with basis $\{g_\lambda\}$. $\MSym$ is isomorphic to $Sym$ as a Hopf algebra, but the basis
$\{g_\lambda\}$ is distinct from the basis of Schur functions, $\{s_\lambda\}$ for $Sym$.

\begin{definition}
 A reverse plane partition T of shape $\lambda$ is a filling of the Young diagram of shape $\lambda$ with positive integers such that the numbers are weakly increasing
in rows and columns.
\end{definition}

Given a reverse plane partition $T$, let $T(i)$ denote the number of columns of $T$ that contain the number $i$. Then 
$x^T:=\displaystyle\prod_{i\in \mathbb{P}}x_i^{T(i)}$. Now we may define the \textit{dual stable Grothendieck polynomial} 
$$g_\lambda=\displaystyle\sum_{sh(T)=\lambda}x^T,$$ where we sum over all reverse plane partitions of shape $\lambda$. For a skew shape $\lambda/\mu$, we may
define $g_{\lambda/\mu}$ analogously, summing over reverse plane partitions of shape $\lambda/\mu$.

\begin{example}
 We use the definition of $g_\lambda$ to compute 
$$g_{(2,1)}=2x_1x_2x_3+2x_1x_3x_4+\ldots+x_1^2x_2+x_1^2x_3+\ldots + x_1^2+x_2^2+\ldots+x_1x_2+x_1x_3+\ldots$$
corresponding 
to fillings of the types shown below. 

\begin{center}
\begin{ytableau}
1 & 2 \\
3 \\
\end{ytableau}\qquad
\begin{ytableau}
1 & 3 \\
2 \\
\end{ytableau}\qquad
\begin{ytableau}
1 & 1 \\
2 \\
\end{ytableau}\qquad
\begin{ytableau}
1 & 1 \\
1 \\
\end{ytableau}\hspace{.2in}
\begin{ytableau}
1 & 2 \\
1 \\
\end{ytableau} 
\end{center}
\end{example}

\subsection{Valued-set tableaux}

We introduce one more basis for $\MSym$, $\{j_\lambda\}$, which we show in the next section is dual under the usual Hall inner product to $\{\tilde J_\lambda\} = \{(-1)^{|\lambda|}J_\lambda(-x_1,-x_2,\ldots)\}$.

\begin{definition}

 A valued-set tableau $T$ of shape $\lambda / \mu$ is a filling of the boxes of $\lambda / \mu$ with positive integers so that
\begin{itemize}
 \item[(1)] the transpose of the filling of $T$ is a semistandard tableau, and
\item[(2)] we are provided with the additional information of a decomposition of the shape into a disjoint union of groups of boxes, $\lambda / \mu = \bigsqcup A_j$,
so that each $A_i$ is connected, contained in a single column, and each box in $A_i$ contains the same number.
\end{itemize}
\end{definition}

Given such a valued-set tableau, $T$, let $a_i$ be the number of groups $A_j$ that contain the number $i$. Then $x^T := \Pi_{i\geq 1} x_i^{a_i}$.
Finally, let $j_{\lambda / \mu}:= \displaystyle\sum_{T}x^T$, where the sum is over all valued-set tableaux of shape $\lambda / \mu$.

\begin{example}
The image below  shows an example of a valued-set tableau. This tableau contributes the monomial $x_1x_2x_3x_5x_6^2$ to $j_{(4,3,1,1)}$. Note that the given assignment of labels will contribute eight monomials---one for each possible decomposition.
\begin{center}
\includegraphics[width=.9in]{valuedset.pdf}
\end{center}
\end{example}

\begin{proposition}\cite[Proposition 9.25]{lam2007combinatorial}
 We have $$\omega(g_{\lambda/\mu})=j_{\lambda/\mu}.$$
\end{proposition}

\section{Antipode results for $\mSym$ and $\MSym$}
As with $\mQSym$ and $\MNSym$, there is a pairing $\langle g_\lambda, G_\mu \rangle=\delta_{\lambda,\mu}$ 
with the usual Hall inner product for $\Sym$ defined by $\langle s_\lambda, s_\mu \rangle=\delta_{\lambda,\mu}$ and the structure 
constants satisfy the conditions of Lemma~\ref{lem:pairingantipode}. See Theorem 9.15 in \cite{lam2007combinatorial} for details. It follows that 
\[\langle \omega(g_\lambda), \omega(G_\mu) \rangle=\langle j_\lambda, \tilde J_\mu \rangle=\delta_{\lambda,\mu}\] and 
\[\langle \tilde j_\lambda, \tilde K_\mu \rangle=
\langle (-1)^{|\lambda|}g_\lambda (-x_1,-x_2,\ldots), (-1)^{|\mu|}G_\mu(-x_1,-x_2,\ldots) \rangle=\delta_{\lambda,\mu}.\]
We will use these facts 
to translate antipode results between $\mSym$ and $\MSym$. 

Using results from Section~\ref{sec:Lhat}, the following lemma will allow us to easily prove results regarding the antipode map in $\mSym$.

\begin{lemma}
Let $\lambda$ be a partition of $n$.
 We can expand $$J_\lambda = \displaystyle\sum_{N\geq n}\sum_{w\in \tilde{\mathcal{J}}_N(P,\theta)}\L_{\omega(\mathcal{C}(w))}.$$
\end{lemma}
\begin{proof}
We know from Theorem~\ref{multiexpandK} that $$\tilde K_{(P,\theta)}=\displaystyle\sum_{N\geq n}\sum_{w\in \tilde{\mathcal{J}}_N(P,\theta)}\tilde L_{\mathcal{C}(w)},$$ 
so $$J_\lambda =\omega(\tilde K_\lambda) = \displaystyle\sum_{N\geq n}\sum_{w\in \tilde{\mathcal{J}}_N(P,\theta)}\omega(\tilde L_{\mathcal{C}(w)}) 
= \displaystyle\sum_{N\geq n}\sum_{w\in \tilde{\mathcal{J}}_N(P,\theta)}\L_{\omega(\mathcal{C}(w))}.$$ 
\end{proof}

Recall that in $Sym$, $S(s_\lambda)=(-1)^{|\lambda|}\omega(s_\lambda)$, so one may expect similar behavior from $\tilde K_\lambda$ and $G_\lambda$. 
Indeed, we obtain the theorem below.

\begin{theorem}\label{thm:mSymAntipode}
In $\mSym$, the antipode map acts as follows.
\begin{itemize}
 \item[(a)] $S(\tilde K_\lambda)=J_\lambda(-x_1,-x_2,\ldots)=(-1)^{|\lambda|}\omega(G_\lambda)$, and
\item[(b)] $S(G_\lambda) = (-1)^{|\lambda|}J_\lambda=(-1)^{|\lambda|}\omega(\tilde K_\lambda)$.
\end{itemize}
\end{theorem}
\begin{proof}
For the first assertion, we have that 
 \begin{eqnarray}
  S(\tilde K_\lambda) &=& S(\displaystyle\sum_{N\geq n}\sum_{w\in \tilde{\mathcal{J}}_N(P,\theta)}\tilde L_{\mathcal{C}(w)}) \nonumber \\
&=& \displaystyle\sum_{N\geq n}\sum_{w\in \tilde{\mathcal{J}}_N(P,\theta)}S(\tilde L_{\mathcal{C}(w)}) \nonumber \\
&=& \displaystyle\sum_{N\geq n}\sum_{w\in \tilde{\mathcal{J}}_N(P,\theta)}\L_{\omega(\mathcal{C}(w))}(-x_1,-x_2,\ldots) \nonumber \\
&=& J_\lambda (-x_1, -x_2,\ldots ). \nonumber
 \end{eqnarray}

And for the second assertion,
\begin{eqnarray}
 S(G_\lambda)&=&S((-1)^{|\lambda|}\tilde K_\lambda (-x_1,-x_2, \ldots)) \nonumber \\
&=&(-1)^{|\lambda|}S(\tilde K_\lambda (-x_1,-x_2,\ldots) \nonumber \\
&=& (-1)^{|\lambda|}J_\lambda. \nonumber
\end{eqnarray}
\end{proof}

By Lemma~\ref{lem:pairingantipode}, we immediately have the following results in $\MSym$.

\begin{theorem}\label{thm:MSymAntipode} 
We have
\begin{itemize}
 \item[(a)] $S(\tilde j_\lambda)= (-1)^{|\lambda|} g_\lambda$, 
where $\tilde j_\lambda = (-1)^{|\lambda|}j_\lambda(-x_1,-x_2,\ldots)$, and 
\item[(b)] $S(j_\lambda)=g_\lambda(-x_1,-x_2,\ldots)$.
\end{itemize} 
\end{theorem}

Next, we work toward expanding $S(G_\lambda)$ and $S(\tilde j_\lambda)$ in terms of $\{G_\mu\}$ and $\{\tilde j_\mu\}$, respectively. We introduce two theorems of Lenart as well as
the notion of a restricted plane partition.

Given partitions $\lambda$ and $\mu$ 
with $\mu\subseteq\lambda$, define an
\textit{elegant filling} of the skew shape $\lambda/\mu$ to be a semistandard filling such that the numbers in row $i$ lie in $[1,i-1]$. In particular, there can be no elegant filling of a shape that has a box in the first row. Now let $f_\lambda^\mu$ denote
the number of elegant fillings of $\lambda/\mu$ for $\mu\subseteq\lambda$ and set $f_\lambda^\mu=0$ otherwise.

\begin{theorem}\cite[Theorem 2.7]{lenart2000combinatorial}\label{sintoG}
For a partition $\lambda$, we have $$s_\lambda = \sum_{\mu \supseteq \lambda} f^{\lambda}_{\mu} G_{\mu},$$ where $f^\lambda_\mu$ is the number of elegant fillings of $\mu/\lambda$.
\end{theorem} 

For the second theorem, let $r_{\lambda \mu}$ be the number of elegant fillings of $\lambda/\mu$ such that both rows and columns are strictly increasing. We will refer to such fillings as \textit{strictly elegant}.

\begin{theorem}\cite[Theorem 2.2]{lenart2000combinatorial}\label{Gintos}
We can expand the stable Grothendieck polynomial $G_\lambda$ in terms of Schur functions as follows
$$G_\lambda=\displaystyle\sum_{\mu \supseteq \lambda}(-1)^{|\mu/\lambda|}r_{\mu\lambda}s_\mu.$$
\end{theorem}

We next define the combinatorial object that we need to expand $S(G_\lambda)$ in terms of $\{G_\mu\}$.

\begin{definition} Let $\lambda\supseteq\mu$ be nonempty partitions. A restricted plane partition is a filling
of the boxes of $\lambda/\mu$ with positive integers such that the entries are weakly decreasing along rows and columns with the following restriction.
If box $b\in\lambda/\mu$ is an outer corner of $\lambda$ (i.e. $\lambda\cup b$ is a partition), define $h(b)$ to be the number of boxes in $\mu$ lying above $b$ in the same column as $b$ or to the left of $b$ in the same row as $b$. The label of such a box $b$ must lie in the interval $[1,h(b)]$.
\end{definition}

We now define the number $P^\mu_\lambda$.
First, $P^\mu_\lambda=0$ if $\mu \nsubseteq \lambda$, and $P^\mu_\lambda=1$ if $\lambda=\mu$. 
If $\mu \subset \lambda$, then $P^\mu_\lambda$ is equal to the number of restricted plane
partitions of the skew shape $\lambda/\mu$.

\begin{example}
The diagram on the left shows $h(b)$ for each box $b$ in the shape $(5,5,5)/(4,2)$ that is an outer corner of $(4,2)$. The diagram on the right shows a restricted plane partition on $(5,5,5)/(4,2)$.

\begin{center}
\begin{ytableau}
 $ $ & $ $ & $ $ & $ $ & 4 \\
$ $ & $ $ & 3 &  &  \\
2 &  &  &  &  \\
\end{ytableau}\hspace{1in}
\begin{ytableau}
 $ $ & $ $ & $ $ & $ $ & 3 \\
$ $ & $ $ & 3 & 3 & 3 \\
2 & 2 & 2 & 1 & 1 \\
\end{ytableau}
\end{center}
\end{example}

\begin{theorem}\label{thm:S(G)} Let $\lambda$ and $\mu$ be partitions. Then
\begin{itemize}
\item[(a)]$S(G_\mu)=(-1)^{|\mu|}\displaystyle\sum_{\lambda}P^{\mu^t}_\lambda G_\lambda$, and 
\item[(b)]$S(\tilde j_\lambda)=(-1)^{|\lambda|} \displaystyle\sum_{\mu}P^{\mu}_{\lambda^t} \tilde j_\mu$.
 \end{itemize}
\end{theorem}

\begin{proof}
 We will focus on part $(a)$, and part $(b)$ will follow from Lemma~\ref{lem:pairingantipode}.

From Theorem \ref{thm:mSymAntipode}, we know that  $$S(G_\lambda) = (-1)^{|\lambda|}J_\lambda,$$ so it remains to expand $J_\lambda$ in terms of stable Grothendieck polynomials. 

From Theorem \ref{Gintos}, it easily follows that we can write $$\tilde K_\lambda=\displaystyle\sum_{\mu \supseteq \lambda}r_{\mu \lambda}s_\mu.$$ Applying $\omega$ to both sides, we have 
 $$\tilde J_\lambda=\displaystyle\sum_{\mu \supseteq \lambda}r_{\mu \lambda}s_{\mu^t}.$$ Now we can use Theorem~\ref{sintoG} to write
$$\tilde J_\lambda =\displaystyle\sum_{\substack{\mu \supseteq \lambda \\ \nu \supseteq \mu^t}}r_{\mu \lambda}f^{\mu^t}_{\nu}G_\nu.$$ Thus the coefficient of $G_\nu$ in $\tilde J_\lambda$ is 
$\displaystyle\sum_{\substack{\mu \text{ such that } \\ \mu\supseteq\lambda \text{ and} \\ \mu^t\subseteq \nu}}r_{\mu\lambda}f^{\mu^t}_{\nu}$.

We describe a bijection between 
\begin{enumerate}
\item partitions of shape $\nu^t/\lambda$ that contain some $\mu\supseteq\lambda$ such that the filling of $\mu/\lambda$ is strictly elegant
and boxes in $\nu^t/\mu$ are filled such that the transpose is an elegant filling of $\nu/\mu^t$ and
\item restricted plane partitions of $\nu^t/\lambda$.
\end{enumerate}
Note that the transpose of restricted plane partition of shape $\nu^t/\lambda$ is a restricted plane partition of shape $\nu/\lambda^t$.

We first define a map $\phi$ that takes objects in group $(1)$ to objects in group $(2)$. Suppose we have such a filling of shape $\nu^t/\lambda$ with some $\mu$ with $\lambda\subset\mu\subset\nu^t$. For any box $b$ in $\nu^t/\lambda$, let $c(b)$ denote the column containing $b$, $r(b)$ denote the row containing $b$, $d(b)=r(b)+c(b)-1$ denote the southwest to northeast diagonal containing $b$, and $e_b$ denote the integer in box $b$.

To obtain a restricted plane partition follow these steps.
\begin{itemize}
 \item[(1)] if box $b$ is in $\mu/\lambda$, fill the corresponding box in the restricted plane partition with $\phi(b)=d(b)-e_b$, and 
\item[(2)] if box $b$ is in $\nu^t/\mu$, fill the corresponding box in the restricted plane partition with $\phi(b)=c(b)-e_b$.
\end{itemize}

\begin{figure}[h]
\begin{center}
\begin{ytableau}
 \cdot & \cdot & \cdot & \cdot \\
\cdot & \cdot & $ $ & b  \\
\cdot & \cdot & $ $ & $ $ \\
$ $ & $ $ \\
\end{ytableau}
\end{center}
\caption{In this figure, boxes in $\lambda$ are marked with a dot. For box $b\in \nu^t/\mu$, we have $c(b)=4$, $r(b)=2$, and $d(b)=5$.}
\end{figure}

It is easy to see that the parts of the restricted plane partition corresponding to shape $\mu/\lambda$ and to $\nu^t/\mu$ are weakly decreasing in rows and columns. 
We now check that entries are weakly
decreasing along the seams and are positive integers. If box $b$ is in $\mu/\lambda$, then $e_b\leq r(b)-1$ because the filling is strictly elegant. Therefore

\[\phi(b)=d(b)-e_b\geq r(b)+c(b)-1 - (r(b)-1)=c(b).\] 
If box $a$ is in $\nu^t/\mu$, then $1\leq e_a \leq c(a)-1$ because the transpose of the filling is elegant, so 

\[1\leq \phi(a)=c(a)-e_a\leq c(a)-1.\]
If $b$ and $a$ are adjacent, then $c(b) \leq c(a)$, so $\phi(b)\geq\phi(a)$.

Next, we check that $\phi(b)\in[1,h(b)]$ for all $b\in \nu^t/\lambda$ that are outer corners of $\lambda$ (i.e. $\lambda\cup b$ is a partition shape). This guarantees that the resulting filling is a restricted plane partition because we have already shown the resulting filling is weakly decreasing. 

Suppose such a box $b$ is in $\mu/\lambda$. Since $b$ is an outer corner of $\lambda$, $d(b)=h(b)+1$. It follows that 

\[\phi(b)=d(b)-e_b = h(b)+1-e_b\leq h(b),\] as desired.

Next suppose box $b$ described above is in $\nu^t/\mu$. Then \[\phi(b)=c(b)-e_b\leq c(b)-1 \leq h(b). \]

Beginning with a restricted plane partition of $\nu^t/\lambda$, we define a map, $\psi$, to recover $\mu$ and the fillings of $\mu/\lambda$ and $\nu^t/\mu$ as follows. Let $b$ be a box in the restricted plane partition. If $e_b\geq c(b)$, then $b$ is in $\mu$ and $\psi(b)=d(b)-e_b$. Note that $e_b\geq c(b)$ guarantees $\psi(b)\leq r(b)-1$, as is required to be strictly elegant.

If $e_b\leq c(b)$, then $b$ is in $\nu^t/\mu$, and $\psi(b)=c(b)-e_b$. Here, $e_b\leq c(b)$ implies that $\psi(b)\leq j-1$, which is necessary to have a transposed elegant filling.

It is easy to see that resulting rows and columns of $\mu$ will be strictly increasing, the resulting rows of $\nu^t/\mu$ will be strictly increasing, and the resulting columns of $\nu^t/\mu$ will be weakly increasing. Thus the image of $\psi$ is a strictly elegant filling of $\mu/\lambda$
and a transposed elegant filling of $\nu^t/\mu$. Clearly the composition of $\phi$ and $\psi$ is the identity, so they are indeed inverses.

\end{proof}

Note that the antipode applied to $G_\lambda$ gives an infinite sum of stable Grothendieck polynomials (see Remark~\ref{rem:Gamma}) while 
applying $S$ to $\tilde j_\lambda$ can be written as a finite sum of $\tilde j$'s. This implies that while the space spanned by stable 
Grothendieck polynomials, $\Gamma$, is not a Hopf algebra, the space spanned by $\tilde j$'s is a Hopf algebra. 

\begin{example}
To illustrate the bijection described above, consider $\lambda=(3,2,1)$, $\mu=(3,3,2,2)$, and $\nu^t=(5,4,4,3)$. The figure on the left is a filling such that $\mu/\lambda$ is strictly elegant and the transpose of $\nu^t/\mu$
is elegant. The entries in $\mu/\lambda$ are in bold. The figure on the right is the corresponding restricted plane partition of $\nu^t/\lambda$.

\begin{center}
\begin{ytableau}
 $ $ & $ $ & $ $ & 2 & 4 \\
$ $ & $ $ & \textbf{1} & 3  \\
$ $ & \textbf{1} & 1 & 3 \\
\textbf{2} & \textbf{3} & 2 \\
\end{ytableau}\hspace{1in}
\begin{ytableau}
 $ $ & $ $ & $ $ & 2 & 1 \\
$ $ & $ $ & 3 & 1  \\
$ $ & 3 & 2 & 1 \\
2 & 2 & 1 \\
\end{ytableau}
\end{center}

 If $b$ is the box in the bottom left corner of the partition on the left, we see that $\phi(b)=d(b)-e_b=4-2=2$. If $a$ is the box in the upper right corner
of the partition on the left, we have $\phi(a)=c(a)-e_a=5-4=1$. In the restricted plane partition on the right, we can see that the boxes in positions $(4,1)$, $(3,2)$, $(4,2)$, and $(2,3)$ are in $\mu/\lambda$ in the image of 
$\psi$ since in these boxes $e_b \geq c(b)$.
\end{example}

\section*{Acknowledgements}
The author thanks Pasha Pylyavskyy and Vic Reiner for many helpful conversations as well as the referee who gave important suggestions to improve clarity.

\bibliographystyle{alpha}
\bibliography{thesis.bib}

\end{document}